\documentclass[10pt,reqno]{amsart}
\usepackage{amsmath,amsthm,amssymb,mathtools,amscd}
\usepackage{latexsym,graphicx,color,indentfirst}
\usepackage{algorithmic}
\usepackage[linesnumbered,boxed,norelsize]{algorithm2e}
\usepackage{fancyhdr}
\usepackage{mathrsfs}
\usepackage{float}
\usepackage{float}
\usepackage{color}
\usepackage{listings}
\usepackage{multirow}
\usepackage{graphicx}
\usepackage{mathrsfs}
\usepackage{ifpdf}
\usepackage{epstopdf}

\newtheorem{theorem}{Theorem}[section]
\newtheorem{lemma}[theorem]{Lemma}

\newtheorem{proposition}[theorem]{Proposition}
\newtheorem{definition}[theorem]{Definition}
\newtheorem{remark}[theorem]{Remark}

\date{\today}

\title[Extract the information from the big data]{Extract the information from the big data with randomly distributed noise}
\author{Cheng Jin$^{1,2}$}
\address{1. School of Mathematical Sciences, Fudan University,
Shanghai 200433, China.}
\address{2. Shanghai Key Labororary of Contempory Appplied Mathematics, Shanghai, China.}
\author{ Zhang Jiantang$^{1}$}
\author{ Zhong Min$^{3,4,*}$}
\address{3. School of Mathematics, Southeast University, 210096, Nanjing, Jiangsu Province.}
\address{4. Nanjing Center for Applied Mathematics, 211135, Nanjing, Jiangsu Province.}
\email{Corresponding author: min.zhong@seu.edu.cn}

\date{}

\begin{document}

\begin{abstract}
In this manuscript, a purely data driven statistical regularization method is proposed for extracting the information from big data with randomly distributed noise. Since the variance of the noise maybe large, the method can be regarded as a general data preprocessing method in ill-posed problems, which is able to overcome the difficulty that the traditional regularization method unable to solve, and has superior advantage in computing efficiency. The unique solvability of the method is proved and a number of conditions are given to characterize the solution. The regularization parameter strategy is discussed and the rigorous upper bound estimation of confidence interval of the error in $L^2$ norm is established. Some numerical examples are provided to illustrate the appropriateness and effectiveness of the method.
\end{abstract}

\maketitle

\section{\bf Introduction}
\setcounter{equation}{0}
As a rapid development of data collection instruments and information technology, the generation of big data arises from applications in different fields as applied mathematics, computer science, geology, biology, engineering, and even business studies. When deal with these big data, there are mainly two types of problems. Firstly, there will always be some redundant data hidden behind, which contains too little effective information or repeated information. Because of the limitation of computing capacity, too much attention on these redundant data will seriously affect the storage and time. Secondly, due to the inevitable measurement errors in the observation, the big data is generated with randomly distributed noise, the variance of these random errors can not be very small. Thus, a direct reconstruction based on these noisy data will lead to unsatisfactory results. To overcome these problems, it is necessary to propose a method satisfying the following requirements: on one hand, the approximation should be accurate and stable, and has strong computing capacity, and low time cost. On the other hand, the method should also guarantee that the rate of change of the objective function is not too large, i.e., the accuracy of derivative.

Determining the function $y(x)$ and the first order or higher order derivatives of $y(x)$ from the random noisy samples of the function values is called numerical differentiation, which has a widely utilization in many practical problems. For example, the determination of discontinuous points in image processing \cite{d83}, the solution of Abel integral equation \cite{gv91}, and some inverse problems arising from
mathematical and physical equations \cite{hs01}, etc. There have been many works concerning the convergence analysis of the numerical algorithms \cite{g91,hs01,rs01}, some different methods have been used to get numerical results \cite{g92,g98,gs93,m87}.

The numerical differentiation is a classical inverse problem in the sense of unstable dependence of solutions on small perturbation of the data, different regularization methods for treating such ill-posed problems in one dimension or higher dimensions were discussed in \cite{ah99,cjw07,d03,dcl12,huang,jwc03,lw04,whw05,wjc,ww05}. However, these methods are not suitable for processing big data with large variation random noise, this is because they are based on accurate knowledge of the noise level $\delta$ or extract this information from the nature of the problem. The prediction or extraction of noiselevel will always over-estimated or under-estimated, thus the reconstruction accuracy is mainly limited.  The following figure describes the difficulties we encountered, obviously, in the context of big data, effective information has been hidden, and traditional regularization methods are unable to handle it.
\begin{figure}[htp]
  \begin{center}
  \includegraphics[width=5.5cm,height=4cm]{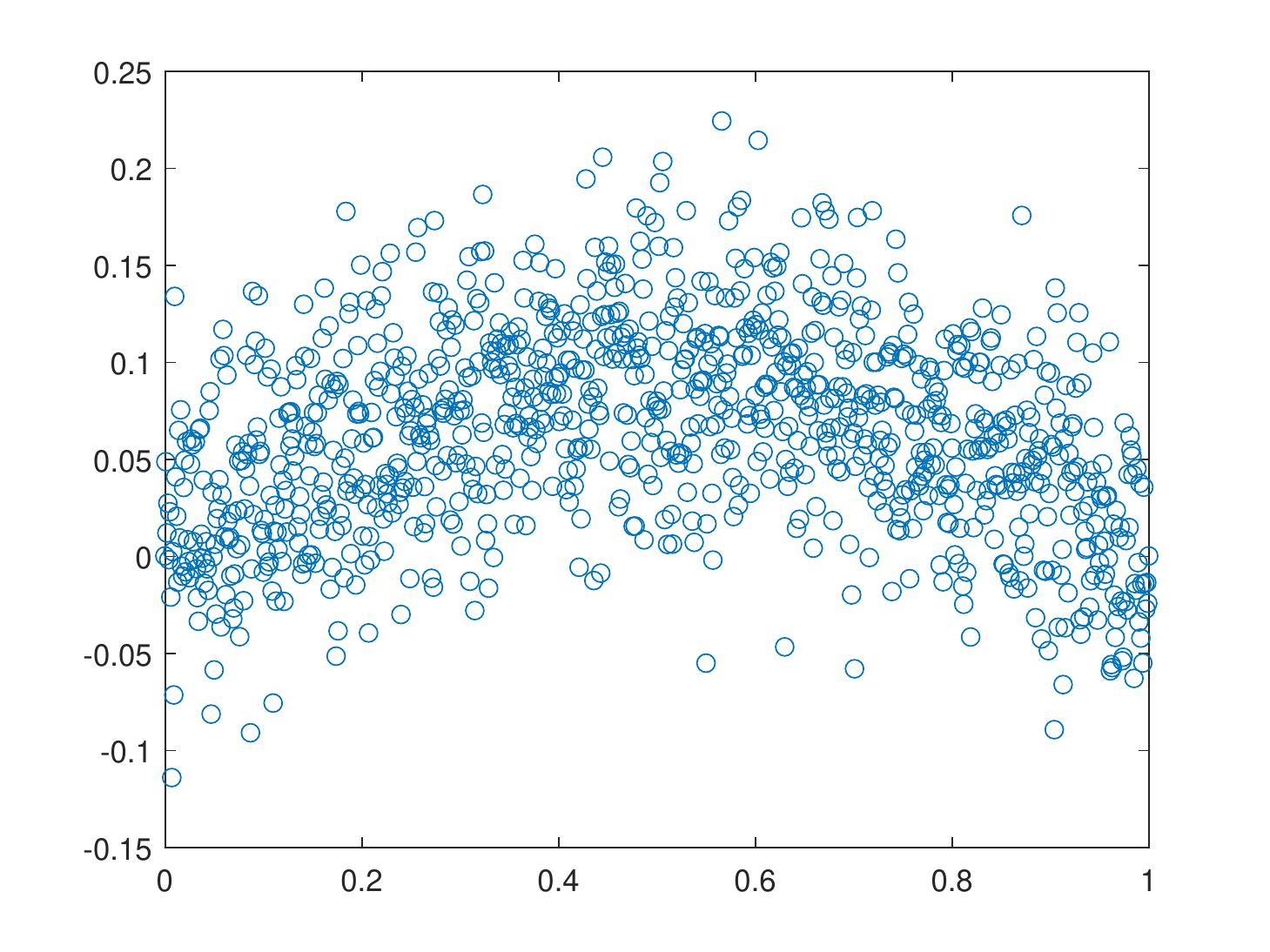}
  \includegraphics[width=5.5cm,height=4cm]{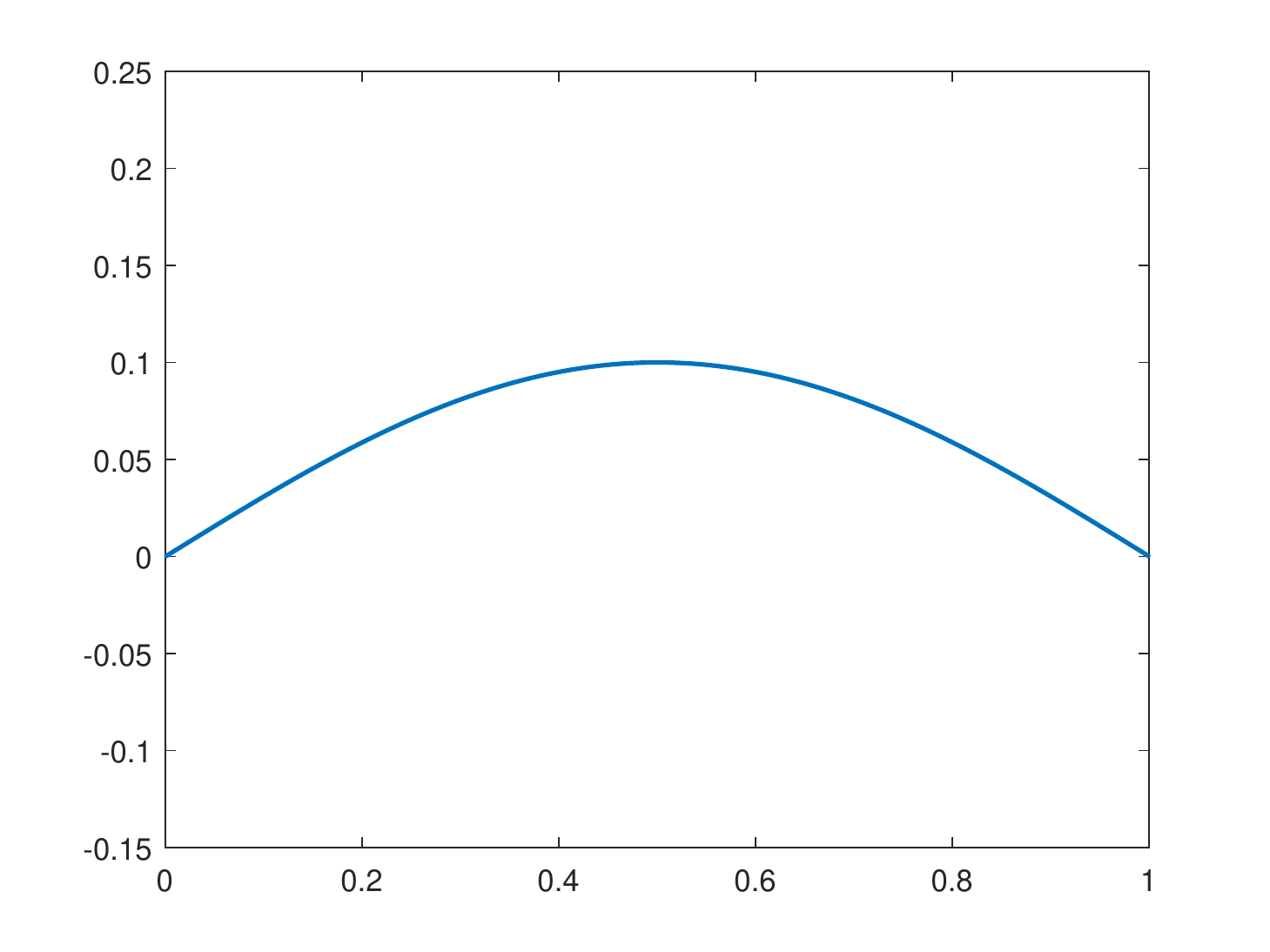}
  \end{center}
  \caption{The noisy samples and exact function for $0.1\times\sin(\pi x)$ with $\sigma^2=0.002$. }
\end{figure}

There are also several statistics data processing techniques \cite{cw78,d83,ss89,w75}, in which the noise is assumed to be independent and identically distributed, the reconstruction results converges to sought solution if the degree of freedom tends to infinity. This greatly increases computational burden when sample size gets large.

Therefore, it is necessary to consider purely data driven numerical method and parameter choice rule. In this manuscript, we propose a new simple statistical method to do the numerical differentiation. One innovation is this method is can overcome the difficulty that the big data with large variation. What we do is to separate the big data into $M$ groups, the average values of the $N$ noisy data in each group is calculated and regarded as a processed data. Based on basic property of distribution, the original variance will be decreased by $N$ times. The other innovation of our method is superior advantage in computing, since the dimension and the size of the big data will be greatly induced after processing. This makes our algorithm run in a relatively small scale, which greatly saves the storage and improves the efficiency. In addition, the algorithm can also be utilized as a data preprocessing method and be applied in different kind of inverse problems.

The manuscript is organize as follows. In section 2, we show the unique solvability of the method and then present a number of conditions characterizing the solution. In section 3, we provide the choice strategy for regularization parameter $\alpha$, and establish the rigorous upper bound estimation of confidence interval of the error in $L^2$ norm, the optimal choice of $M$ and $N$ will be discussed as well. The algorithm was given in section 4, and some numerical examples are provided to illustrate the effectiveness and the computational performance in section 5.

\section{\bf Algorithm description and some basic results}
\setcounter{equation}{0}
\subsection{Formulation of the problem}
Suppose $y = y(x)$ is a function defined on $[0,1]$, and $0=\hat x_0<\hat x_1<\ldots<\hat x_L=1$ is a uniform grid on $[0,1]$ with meshsize $h=1/L$. Given the noisy observation
samples $\widetilde{y}_j=y(\hat x_j)+\eta_j$, with the error
\begin{align}\label{white noise}
\eta_j\sim\mathcal{N}(0,\sigma^2)\,,
\end{align}
where the $\sigma^2$ denotes the variance of observation noise, we are interested in reconstructing
a function $f(x)$ such that the derivatives of $f(x)$ are approximations of derivatives
of function $y(x)$. Without losing the generality, we may assume the observations on endpoints are exact, i.e.,
\begin{align*}
\widetilde {y}_0 = y(0)\,,\quad\widetilde{y}_L = y(1)\,.
\end{align*}

The main difficulties we encounter in solving such ill-posed numerical differential problem are the observation size $L$ can be very large, and the variance $\sigma^2$ can be large as well.
To overcome these problems, given a positive integer $M>1$, we separate the big data into $M$ groups. In each group, there are $N=L/M$ observation points, i.e.,
\begin{align}\label{G}
&G_1 := \{\hat x_1<\hat x_2<\ldots<\hat x_N\}\,,\nonumber \\
&G_2 := \{\hat x_{N+1}<\hat x_{N+2}<\ldots<\hat x_{2\times N}\}\,,\nonumber\\
&\ldots\nonumber \\
&G_M := \{\hat x_{(M-1)\times N+1}< \hat x_{(M-1)\times N+2}<\ldots< \hat x_{M\times N}\}\,.
\end{align}
Then, we denote the new uniform grid
\begin{align}\label{newgrid}
\triangle = \{ 0=x_0<x_1< \ldots < x_M=1\}\,
\end{align}
with $ x_i = \hat x_{iN} (0\leq i\leq M)$ and corresponding mesh size $h_M = 1/M$.
The new observation samples $\widetilde{Y}_i$ is defined as sample mean of the group $G_i$:
\begin{align}\label{new sample}
\widetilde{Y}_i = \frac{\widetilde y_{(i-1)\times N+1}+\cdots +\widetilde y_{i\times N}}{N}\,.
\end{align}

Let $k>1$ be an integer and $M,N>k$, we denote
\begin{align*}
\mathcal{H}: = \{f|f\in H^k(0,1),\ f(0) = y(0),\ f(1) = y(1)\}\,,
\end{align*}
in which $H^k(0,1)$ be the usual Sobolev space consisting of all $L^2(0,1)-$integrable functions whose $k-$order weak derivatives are also $L^2(0,1)$ integrable.  Define the cost functional
\begin{align}\label{functional}
\Phi(f): =\frac{1}{M}\sum_{i=1}^M\left(\widetilde{Y}_i-M_i(f)\right)^2+\alpha\|f^{(k)}\|_{L^2(0,1)}^2\,,
\end{align}
where $\alpha>0$ is regularization parameter, and the
\begin{align*}
M_i(f) = \frac{1}{h_M}\int_{x_{i-1}}^{x_i} f(x) dx\,
\end{align*}
is the average value of the function $f$ over the interval $[x_{i-1},x_i]$. It can be shown there exists a unique minimizer $f_*\in\mathcal{H}$ for the functional $\Phi$. The minimizer will be solutions of the numerical differential problem. In following section 3, a number of conditions are given to characterize the solution.
\begin{remark}
The similar Tikhonov minimization functional can be found in \cite{huang}, in which the author also utilized the average value in the data fitting term, but $\|f'\|_{L^2(0,1)}$ in penalty term.
\end{remark}
\subsection{Some results in statistics}
In this subsection, we introduce some basic results in statistics.
\begin{definition}
We call elementary probability space a triplet $(\Omega,\mathcal{F},P)$ where $\Omega$ is a set. $\mathcal{F}$ is an algebra of subsets of $\Omega$ and $P$ is a set of function defined on $\mathbb{R}$ with values in the real interval $[0,1]$, which satisfies the following relations
\begin{enumerate}
  \item $P(\Omega)=1$,
  \item if $A_1,A_2,\ldots,A_n\in\mathcal{F}$ and $A_i\bigcap A_j = \emptyset$ for $i\neq j$, then
  \begin{align*}
  P(\bigcup_{k=1}^n A_k) = \sum_{k=1}^n P(A_k)\,.
  \end{align*}
\end{enumerate}
\end{definition}

\begin{definition}
Given the probability space $(\Omega,\mathcal{F},P)$, we call a random variable a real valued function $\psi:\Omega\rightarrow\mathbb{R}$ such that $\psi^{-1}(B)\in\mathcal{F}$ for every Borel set $B$.
\end{definition}

\begin{theorem}\label{levy}[Lindeberg-Levy: Central Limit Theorem for {\bf iid} variables]
Let $X_1,X_2,\cdots,X_n$ are independent identically distributed ({\bf iid}) random variables with finite mean $\mu$ and variance $\sigma^2$, then the random variable
\begin{align*}
Y_n = \frac{\sum_{i=1}^n X_i-n\mu}{\sigma\sqrt{n}}\rightarrow Z\,,\quad \textrm{as}\ n\rightarrow\infty\,,\ \textrm{in distrbution}\,,
\end{align*}
where the $Z$ denotes the standardized Gaussian random variable whose probability density function is
\begin{align*}
F_Z(x) = \frac{1}{\sqrt{2\pi}}\int_\infty^x e^{-t^2/2}dx\,.
\end{align*}
\end{theorem}

The theorem can be proving by recalling Levy's theorem and using character functions. In different words, the statement of theorem can be expressed by saying that the variable $\sum_{i=1}^N X_i$ is asymptotically Gaussian with mean $n\mu$ and variance $n\sigma^2$. This is ,in turn, implies that the sample mean $\overline{X}: = \frac{1}{n}\sum_{i=1}^n X_i$ is asymptotically normal with mean $E(\overline{X}) = \mu$ and variance $D(\overline{X}) = \frac{\sigma^2}{n}$ and standard deviation $\frac{\sigma}{\sqrt n}$. This fact, we will see, often plays an important role in cases where a large number of elements is involved.

\begin{theorem}\label{Markov}[Markov's inequality]
If $X$ is any nonnegative random variable and $a>0$, then
\begin{align}\label{M1}
P(X\geq a) \leq\frac{E(X)}{a}\,.
\end{align}
More generally, if $\varphi:[0,+\infty)\rightarrow\mathbb{R}$ is a nonnegative increasing function for $a>0$ and $\varphi(a)>0$, we have
\begin{align}\label{M2}
P(X\geq a) = P(\varphi(X)\geq\varphi(a))\leq\frac{E(\varphi(X))}{\varphi(a)}\,.
\end{align}
\end{theorem}

We also need definition and properties for chi-squared distribution which was based on the definition of Gamma distribution.
\begin{definition}
Given real numbers $r>0$ and $\lambda>0$, the random variable $X$ is said to have the gamma probability density function with parameters $r$ and $\lambda$ if
\begin{align*}
f_X(x) = \frac{\lambda^r}{\Gamma(r)}x^{r-1}e^{-\lambda x}\,,\quad y\geq0\,,
\end{align*}
in which the $\Gamma(r)$ is gamma function defined by
\begin{align*}
\Gamma(r) = \int_0^\infty x^{r-1}e^{-x}dx
\end{align*}
with properties
\begin{enumerate}
  \item $\Gamma(1)=1$\,,
  \item $\Gamma(r) = (r-1)\Gamma(r-1)$\,,
  \item If $r$ is an integer, $\Gamma(r) = (r-1)!$\,.
\end{enumerate}
\end{definition}

A generalized gamma probability density function plays major role in statistics, and it can have some different types. Among the most common of all statistical analyses are procedures known as $\chi^2$ (chi-squared) distribution, which is a special case of the gamma probability density function  with $\lambda=1/2$ and $r=m/2$, and $m$ is a positive integer denotes the number of degrees of freedom.

\begin{proposition}\label{chisquare}
Assume $X_1,\ldots,X_n$ be {\bf iid} random variables, $X_i\sim \mathcal{N}(0,1)$. Then the random variable $\chi^2 = \sum_{i=1}^n X_i^2$ has a chi-square distribution with $n$ degrees of freedom denoted by $\chi^2\sim\chi^2(n)$. The $\chi^2-$ distribution have several properties
\begin{enumerate}
  \item The mean and the variance of the $\chi^2(n)$distribution are $n$ and $2n$\,,
  \item If $X\sim\chi^2(n_1)$, $Y\sim\chi^2(n_2)$ and $X,Y$ are independent variables, then $X+Y\sim\chi^2(n_1+n_2)$\,,
  \item The Cumulative distribution function is
  \begin{align*}
  F(x,n) = \frac{\gamma(\frac{n}{2},\frac{x}{2})}{\Gamma(\frac{n}{2})}
  \end{align*}
  in which $\gamma$ is incomplete gamma function defined as
  \begin{align*}
  \gamma\left(\frac{n}{2},\frac{x}{2}\right) = \int_0^{x/2}t^{n/2-1}e^{-t}dt\,.
  \end{align*}
\end{enumerate}
\end{proposition}


\section{Main theoretical results}
\setcounter{equation}{0}

For the cost functional $\Phi$ defined in \eqref{functional}, we have
\begin{theorem}\label{th1}
There exists one unique function $f_*\in\mathcal{H}$ such that
\begin{align*}
\Phi(f_*)\leq\Phi(f)
\end{align*}
for any functions $f\in\mathcal{H}$.
\end{theorem}
\begin{proof}
{\bf Step 1} Construction of the minimizer $f_*$.

First we assume there exists a minimizer $f_*\in\mathcal{H}$. Define a function $F(\lambda) = \Phi(f_*+\lambda g)$ where $g(x)\in H^k$ satisfies $g(0)=g(1)=0$. Due to the minimality of $f_*$, we have $F'(0)=0$. Since
\begin{align*}
\Phi(f_*+\lambda g) &= \frac{1}{M}\sum_{i=1}^M\left(\widetilde{Y}_i-M_i(f_*+\lambda g)\right)^2+\alpha\|f_*^{(k)}+\lambda g^{(k)}\|_{L^2(0,1)}^2\\
&=\frac{1}{M}\sum_{i=1}^M\left(\left(\widetilde{Y}_i-M_i(f_*)\right)^2-2\lambda\left(\widetilde{Y}_i-M_i(f_*)\right)M_i(g)
+\lambda^2\left(M_i(g)\right)^2\right)\\
&\quad+\alpha\left(\|f_*^{(k)}\|^2+2\lambda\int_0^1f_*^{(k)}g^{(k)}dx+\lambda^2\|g^{(k)}\|^2\right)\,.
\end{align*}
Therefore,
\begin{align}\label{stable}
F'(0) = \frac{1}{M}\sum_{i=1}^M\left(-2\left(\widetilde{Y}_i-M_i(f_*)\right)M_i(g)\right)+2\alpha\int_0^1f_*^{(k)}g^{(k)}dx=0\,.
\end{align}
Since
\begin{align*}
\int_0^1f_*^{(k)}g^{(k)}dx &= f_*^{(k)}g^{(k-1)}|_0^1-\int_0^1f_*^{(k+1)}g^{(k-1)}dx\\
&=\ldots = \sum_{j=0}^{k-1}(-1)^jf_*^{(k+j)}g^{(k-1-j)}|_0^1+(-1)^k\int_0^1f_*^{(2k)}gdx\,,
\end{align*}
in addition, $g\in H^k$ can be arbitrary, it can be verified that $f_*\in H^{2k-1}(0,1)$ \cite{adams} and $f_*$ satisfies the following boundary conditions:
\begin{align}\label{condition1}
f_*^{(k+j)}(0) = f_*^{(k+j)}(1)=0\,,\quad j=0,1,\ldots,k-2\,.
\end{align}
Therefore,
\begin{align*}
F'(0) &= \frac{1}{M}\sum_{i=1}^M\left(-2\left(\widetilde{Y}_i-M_i(f_*)\right)M_i(g)\right) + 2\alpha(-1)^k\int_0^1 f_*^{(2k)}gdx \\
& = -2\sum_{i=1}^M\int_{x_{i-1}}^{x_i}\left(\widetilde{Y}_i-M_i(f_*)\right)gdx+ 2\alpha(-1)^k\int_0^1 f_*^{(2k)}gdx=0\,.
\end{align*}

We plug any function $g\in C_0^{\infty}(x_{i-1},x_i)$ into the equation, it is equivalently
\begin{align*}
\int_{x_{i-1}}^{x_i}\left[\left(M_i(f_*)-\widetilde{Y}_i\right)+(-1)^k\alpha f_*^{(2k)}\right]g dx = 0\,.
\end{align*}
Consequently,
\begin{align}\label{condition2}
f_*^{(2k)} = \frac{(-1)^k}{\alpha}\left(\widetilde{Y}_i-M_i(f_*)\right)\,,\quad\forall x\in(x_{i-1},x_i)
\end{align}
which is a constant. It can be concluded that $f_*(x)|_{(x_{i-1},x_i)}\in P_{2k}(x_{i-1},x_i)$, where $ P_{2k}(x_{i-1},x_i)$ represents the set of all polynomials of one variable with the order no more than $2k$.

Considering the term $\int_0^1f_*^{(k)}g^{(k)}dx$ by integration by parts again, for all $g\in C_0^\infty(0,1)$, since
\begin{align*}
\int_0^1 f_*^{(k)}g^{(k)}dx &= \sum_{i=1}^M\int_{x_{i-1}}^{x_i} f_*^{(k)}g^{(k)}dx\\
&=\sum_{i=1}^M\left(\sum_{j=0}^{k-1}(-1)^j f_*^{(k+j)}g^{(k-1-j)}|_{x_{i-1}}^{x_i}\right)+(-1)^k\sum_{i=1}^M\int_{x_{i-1}}^{x_i} f_*^{(2k)}gdx\,.
\end{align*}
Substituting this into \eqref{stable} and noting equation \eqref{condition2}, we see
\begin{align*}
\sum_{i=1}^M\left(\sum_{j=0}^{k-1}(-1)^j f_*^{(k+j)}g^{(k-1-j)}|_{x_{i-1}}^{x_i}\right)=0\,.
\end{align*}
Noting the condition \eqref{condition1} and $g(0)=g(1)=0$, it is equivalent that
\begin{align*}
\sum_{j=0}^{k-1}(-1)^j\left(\sum_{i=1}^{M-1}
\left(f_*^{(k+j)}(x_i-0)-f_*^{(k+j)}(x_i+0)\right)g^{(k-1-j)}(x_i)\right)=0\,,
\end{align*}
which implies
\begin{align}\label{condition3}
f_*^{(k+j)}(x_i+0)=f_*^{(k+j)}(x_i-0)\,,\quad j=0,\ldots,k-1\,,\quad i=1,\ldots,M-1\,.
\end{align}

Furthermore, observing that $f_*\in H^{2k-1}(0,1)$ which can be continuously embedded into $C^{2k-2}[0,1]$, we add another $k$ conditions that
\begin{align}\label{condition4}
f_*^{(j)}(x_i+0)=f_*^{(j)}(x_i-0)\,,\quad j=0,\ldots,k-1\,,\quad i=1,\ldots,M-1\,.
\end{align}

Combining the conditions \eqref{condition1}-\eqref{condition4}, the solution $f_*(x)$ given as
\begin{align}\label{polynomial}
f_*(x) = c_1^i+c_2^ix+\cdots+c_{2k+1}^ix^{2k}\,,\quad x\in(x_{i-1},x_i),\quad, i=1,2,\ldots M\,,
\end{align}
where the $(2k+1)M$ parameters $c_k^i$, $k=1,2,\ldots 2k+1$, $i=1,2,\ldots, M$, are determined by solving the following $(2k+1)M$ linear equations:
\begin{align}\label{formula}
f_*^{(j)}(x_i+) - f_*^{(j)}(x_i-)&=0\,,\quad j=0,1,\ldots, 2k-1,\ i=1,2,\ldots, M-1\,,\nonumber\\
f_*^{(2k)}(x) &= \frac{(-1)^k}{\alpha}\left(\widetilde{Y}_i- M_i(f_*)\right)\,,\quad x\in(x_{i-1},x_i)\,,\quad i=1,2,\ldots, M\,,\nonumber\\
f_*^{(k+j)}(0)&=0\,,\quad j=0,1,\ldots, k-2\,,\nonumber\\
f_*^{(k+j)}(1)&=0\,,\quad j=0,1,\ldots, k-2\,,\nonumber\\
f(0)&=y(0)\,,\quad f(1)=y(1)\,.
\end{align}

Next, we will prove that these linear equations with respect the unknown constants are uniquely solvable. We only need to show that the homogenous equations have only the trivial solution. If we choose $\widetilde{Y}_i=0$ for $i=1,2,\ldots,M$, then the linear equations we obtained are just the homogenous equations. On the other hand, by the definition of functional $\Phi$, we know that $f_* = 0$ is the unique minimzer. This means that the homogenous equation have only the trivial solution.

{\bf Step 2:} The uniqueness of the minimizer $f_*$.

For arbitrary $f\in\mathcal{H}$, we denote $g(x) = f(x)-f_*(x)$, it is obvious that $g(0)=g(1)=0$ and
\begin{align*}
\Phi(f)-\Phi(f_*) &= \frac{1}{M}\left[\sum_{i=1}^M\left(\widetilde{Y}_i-M_i(f)\right)^2-\left(\widetilde{Y}_i-M_i(f_*)\right)^2\right]
+\alpha\left(\|f^{(k)}\|^2-\|f_*^{(k)}\|^2\right)\\
&=\frac{1}{M}\sum_{i=1}^M\left(M_i(f_*)-M_i(f)\right)\left(2\widetilde{Y}_i-M_i(f)-M_i(f_*)\right)
+\alpha\left(\|f^{(k)}\|^2-\|f_*^{(k)}\|^2\right)\,.
\end{align*}
It is worth to note that
\begin{align*}
\|f^{(k)}\|^2-\|f_*^{(k)}\|^2&=\|f^{(k)}-f_*^{(k)}\|^2+2\int_0^1\left(f^{(k)}-f_*^{(k)}\right)f_*^{(k)}dx\\
&=\|g^{(k)}\|^2+2\sum_{j=0}^{k-1}(-1)^jf_*^{(k+j)}g^{(k-1-j)}|_0^1+2(-1)^k\int_0^1f_*^{(2k)}gdx\\
&=\|g^{(k)}\|^2+2(-1)^k\sum_{i=1}^M\int_{x_{i-1}}^{x_i}f_*^{(2k)}gdx\\
&=\|g^{(k)}\|^2+2\sum_{i=1}^M\int_{x_{i-1}}^{x_i}\frac{1}{\alpha}\left(\widetilde{Y}_i-M_i(f_*)\right)gdx\\
& = \|g^{(k)}\|^2+\frac{2}{M}\sum_{i=1}^M\frac{1}{\alpha}\left(\widetilde{Y}_i-M_i(f_*)\right)M_i(g)\,.
\end{align*}
Therefore,
\begin{align*}
\Phi(f)-\Phi(f_*) &= \frac{1}{M}\sum_{i=1}^M\left(M_i(f_*)-M_i(f)\right)\left(2\widetilde{Y}_i-M_i(f)-M_i(f_*)\right)
\\
&\quad+\alpha\|g^{(k)}\|^2+\frac{2}{M}\sum_{i=1}^M\left(\widetilde{Y}_i-M_i(f_*)\right)M_i(g)\\
& = \alpha\|g^{(k)}\|^2 + \frac{1}{M}\sum_{i=1}^M\left(M_i(g)\right)^2\geq 0\,.
\end{align*}
It means that the function $f_*$ is a minimizer of the functional $\Phi(f)$. If there is another function $f_1$ such that $\Phi(f_*) = \Phi(f_1)$, then following the above process, we may have $f_*^{(k)} = f_1^{(k)}$ and thus $f_*-f_1$ is a polynomial of degree $k-1$. In addition since $M_i(f_*-f_1) = 0$, implying there exists $\xi\in(x_{i-1},x_i)$ such that $f_*(\xi)-f_1(\xi) = 0$, i.e., there exists at least one root in $(x_{i-1},x_i)$ and there exist at lease $M$ roots in $(0,1)$. It is a contradiction provided that $M>k$. Therefore, $f_1 = f_*$ which means that the minimizer $f_*$ is unique.

The proof is complete.
\end{proof}
\begin{remark}
For the proof of Theorem \ref{th1}, we know that $f_*\in\mathcal{H}$ is a piecewise polynomial of degree $2k$ which can be determined by $(2k+1)M$ parameters.
\end{remark}

 In order to obtain the error estimate for the algorithm, we first define an interpolation operator $Q_{h_M}$ from $L^2(0,1)$ onto the space of step functions related to the subdivision $\triangle$ as follows. For a function $g\in L^2(0,1)$, $Q_{h_M}g$ is given by
 \begin{align}\label{Q_{h_M}}
 Q_{h_M}g(x) = M_i(g),\quad\forall x\in (x_{i-1},x_i)\,.
 \end{align}
 We can obtain the error estimate for the operator $Q_{h_M}$ by the usual scaling argument \cite{c78,hx98}, as described in the following result.
 \begin{lemma}\label{lem Q_{h_M}}
 For all $g\in H^1(0,1)$,
 \begin{align*}
 \|g-Q_{h_M}g\|_{L^2(0,1)}\leq h\|g'\|_{L^2(0,1)}\,.
 \end{align*}
 \end{lemma}

\begin{lemma}
Let $y$ be a function in $H^k(0,1)$, and let $f_*$ be the solution of the method \eqref{minimization}. Denote $e= f_*-y$, and denote the variable
\begin{align}\label{triangle}
\triangle_M^2 = \frac{1}{M}\sum_{i=1}^M(\widetilde Y_i-Y_i)^2\,.
\end{align}
Then, we have the following two estimates
\begin{align}\label{estimate1}
\|e^{(k)}\|\leq \sqrt{\frac{2\triangle_M^2}{\alpha}+\frac{2Q^2h_M^2}{\alpha N^2}} + 2\|y^{(k)}\|\,,
\end{align}
and
\begin{align}\label{estimate2}
\|e\|\leq h\|e'\| + \sqrt{8\triangle_M^2+\frac{8Q^2h_M^2}{N^2}+2\alpha\|y^{(k)}\|^2}\,,
\end{align}
in which $Q$ denote the upper bound of the $\|f'\|_{L^2(0,1)}$ for all $f\in H^k(0,1)$.
\end{lemma}
\begin{proof}
Putting $y\in\mathcal{H}$ as candidate into the functional $\Phi$ in \eqref{functional}, by the minimality of the $f_*$, we have
\begin{align*}
\Phi(f_*)\leq\Phi(y) &=\frac{1}{M}\sum_{i=1}^M\left(\widetilde{Y}_i-M_i(y)\right)^2+\alpha\|y^{(k)}\|^2\\
&\leq \frac{2}{M}\sum_{i=1}^M\left((\widetilde{Y}_i-Y_i)^2+\left(Y_i-M_i(y)\right)^2\right)+\alpha\|y^{(k)}\|^2\,.
\end{align*}
Since
\begin{align*}
|Y_i-M_i(y)| &= \left|\frac{1}{h_M}\int_{x_{i-1}}^{x_i}ydx - \frac{y_{(i-1)\times N+1}+\cdots+y_{i\times N}}{N}\right|\\
&= \left|\frac{1}{h_M}\int_{x_{i-1}}^{\hat x_{(i-1)\times N+1}}ydx-\frac{y(\hat x_{(i-1)\times N+1})}{N}+\sum_{j=2}^N\left(\frac{1}{h_M}\int_{\hat x_{(i-1)\times N+j-1}}^{\hat x_{(i-1)\times N+j}}ydx-\frac{y(\hat x_{(i-1)\times N+j})}{N}\right)\right|\\
& \leq \frac{1}{N}\sum_{j=1}^N\left|y(\xi_j) - y(x_{\hat (i-1)\times N+j})\right|\\
& = \frac{1}{N}\sum_{j=1}^N\left|y(\hat x_{(i-1)\times N+j})+(\xi_j-\hat x_{(i-1)\times N+j})y'(\eta_j)-y(\hat x_{(i-1)\times N+j})\right|\\
&\leq \frac{Qh_M}{N}\,,
\end{align*}
in which $\xi_1\in(x_{i-1},\hat x_{(i-1)\times N+1}),\ \xi_j\in(\hat x_{(i-1)\times N+j-1},\hat x_{(i-1)\times N+j})\ \textrm{for}\ j\geq 2$ and $\eta_j\in(\xi_j,\hat x_{(i-1)\times N+j})$, this yields
\begin{align}\label{add}
\frac{2}{M}\sum_{i=1}^M |Y_i-M_i(y)|^2\leq \frac{2Q^2h_M^2}{N^2}\,.
\end{align}
Therefore,
\begin{align}\label{minimality}
\alpha \|(f_*)^{(k)}\|^2\leq \Phi(f_*) &= \frac{1}{M}\sum_{i=1}^M\left(\widetilde{Y}_i-M_i(f_*)\right)^2+\alpha\|f_*^{(k)}\|^2\nonumber\\
&\leq 2\triangle_M^2+\frac{2Q^2h_M^2}{N^2} + \alpha\|y^{(k)}\|^2\,.
\end{align}
Form this, it is subsequently that
\begin{align*}
\|e^{(k)}\|\leq\|f_*^{(k)}\|+\|y^{(k)}\|\leq \sqrt{\frac{2\triangle_M^2}{\alpha}+\frac{2Q^2h_M^2}{\alpha N^2}} + 2\|y^{(k)}\|\,.
\end{align*}

On the other hand, noting that $Q_{h_M}$ is a projection operator, we obtain
\begin{align*}
\|e\|^2 = \int_0^1e^2dx &= \int_0^1e(e-Q_{h_M}e)dx+\int_0^1(Q_{h_M}e)^2dx\\
&: = I_1+I_2\,.
\end{align*}
The $I_1$ can be estimated by using the Cauchy-Schwartz inequality and Lemma \ref{lem Q_{h_M}},
\begin{align*}
\left|\int_0^1e(e-Q_{h_M}e)dx\right|\leq\|e\|\|e-Q_{h_M}e\|\leq h\|e\|\|e'\|\,.
\end{align*}
For the $I_2$, referring to \eqref{add} and \eqref{minimality} again, we have that
\begin{align*}
\left|\int_0^1(Q_{h_M}e)^2dx\right| &= \sum_{i=1}^M\int_{x_{i-1}}^{x_i}(Q_{h_M}e)^2dx = h_M\sum_{i=1}^MM_i(e)^2\\
& =  h_M\sum_{i=1}^M\left(M_i(y)-M_i(f_*)\right)^2\\
&\leq \frac{2}{M}\left(\sum_{i=1}^M(M_i(y)-\widetilde{Y}_i)^2+ \sum_{i=1}^M(\widetilde{Y}_i-M_i(f_*))^2\right)\\
&\leq \frac{4}{M}\left(\sum_{i=1}^M\left(M_i(y)-Y_i\right)^2\right)+4\triangle_M^2p+4\triangle_M^2+\frac{4Q^2h_M^2}{N^2} + 2\alpha\|y^{(k)}\|^2\\
& = 8\triangle_M^2+\frac{8Q^2h_M^2}{N^2} + 2\alpha\|y^{(k)}\|^2\,.
\end{align*}
Combining the above two estimates, we have
\begin{align*}
\|e\|^2\leq h\|e\|\|e'\| + 8\triangle_M^2+\frac{8Q^2h_M^2}{N^2} + 2\alpha\|y^{(k)}\|^2\,,
\end{align*}
hence
\begin{align*}
\|e\|\leq h\|e'\| + \sqrt{8\triangle_M^2+\frac{8Q^2h_M^2}{N^2}+2\alpha\|y^{(k)}\|^2}\,.
\end{align*}
\end{proof}

The above lemma shows that, when the regularization parameter $\alpha$ is given, the estimates for the $\|e\|_{L^2(0,1)}$ and $\|e^{(k)}\|_{L^2(0,1)}$ can be determined and controlled by the statistics $\triangle_M^2$, which depend on the random statistics $\widetilde{Y}_i$. Since
\begin{align*}
\widetilde{Y}_i-Y_i\sim\mathcal{N}(0,\frac{\sigma^2}{N})\,,
\end{align*}
and due to the independence of the $\widetilde{Y}_i$, referring to Proposition \ref{chisquare}, we know
\begin{align*}
\frac{MN}{\sigma^2}\triangle_M^2\sim\chi^2(M)\,.
\end{align*}
For a fixed $p\in(0,1)$, we denote the $(1-p)-$quantile of the $\chi^2(M)$ be $Z_{M,p}$, we provide the following lemma:
\begin{lemma}\label{chi}
For $0<p<0.37$, there exists an upper bound for the estimate of $\frac{Z_{M,p}}{M}$, denoted by $\bar{z}_{M,p}$, which can be determined by the unique solution of the equation
\begin{align*}
xe^{1-x}-p^{\frac{2}{M}}=0
\end{align*}
in $(1,+\infty)$. In addition, the $\bar{z}_{M,p}$ satisfies
\begin{align}\label{property}
\bar{z}_{M-1,p}>\bar{z}_{M,p}\,,\ \textrm{and}\ \lim_{M\rightarrow\infty}\bar{z}_{M,p}=1\,.
\end{align}
\end{lemma}
\begin{proof}
Recalling the Markov inequality \eqref{M2}, let $\varphi(X) = \exp(tX)$, we have
\begin{align}\label{M3}
P(X\geq a) = P(\exp(tX)\geq\exp(ta))\leq\frac{\mathbb{E}(\exp(tX))}{\exp(ta)}\,,\quad\forall\ t>0\,,
\end{align}
in which the numerator of the right hand side of \eqref{M3} is the moment generating function. If the random variable $X\sim\chi^2(M)$, the moment generating function is defined by
\begin{align*}
\mathbb{E}(\exp(tX)) = (1-2t)^{-\frac{M}{2}}\,.
\end{align*}
Let $a = \frac{Z_{M,p}}{M}: = z_{M,p}$, where $Z_{M,p}$ is the $(1-p)-$quantile of the $\chi^2(M)$ distribution, the \eqref{M3} becomes
\begin{align*}
p = P(X\geq z_{M,p}M)\leq \frac{(1-2t)^{-\frac{M}{2}}}{\exp(tz_{M,p}M)}\,,\quad\forall t>0\,.
\end{align*}
Choosing $t = (z_{M,p}-1)/2z_{M,p}$, we have
\begin{align}\label{M4}
p = P(X\geq z_{M,p}M)\leq \left(z_{M,p}\exp(1-z_{M,p})\right)^{\frac{M}{2}}\,.
\end{align}

Considering the function $\psi(x) = x\exp(1-x)$, it is obvious that $\psi(x)$ is strictly decreasing in $(1,+\infty)$ and $\psi(1) = 1$, $\lim_{x\rightarrow+\infty}\psi(x)=0$, therefore, the equation
\begin{align*}
\psi(x) = p^{\frac{2}{M}}
\end{align*}
exists a unique solution $\overline{z}_{M,p}\in(1,+\infty)$. Recalling \eqref{M4}, the $\overline{z}_{M,p}$ can be regarded as the upper bound for $z_{M,p}$, this is because
\begin{align*}
p = \left(\overline{z}_{M,p}\exp(1-\overline{z}_{M,p})\right)^{\frac{M}{2}}\leq \left(z_{M,p}\exp(1-z_{M,p})\right)^{\frac{M}{2}}\,,
\end{align*}
which implies
\begin{align*}
z_{M,p}\leq \overline{z}_{M,p}\,.
\end{align*}
In addition, $p^{2/M}$ is increased with respect to $M$, this combining with the decreasing property of $\psi(x)$, we known $\overline{z}_{M,p}$ is also decreased with respect to $M$, which yielding \eqref{property}.
Considering the cumulative distribution function in Proposition \ref{chisquare}, it is necessary demand $F(z_{M,p}M,M)\geq F(M,M)$ such that $z_{M,p}\geq 1$. Since $F(M,M)$ is decreased with respect to $M$, and $M$ is a positive integer with $M\geq 2$. Therefore, we can demand $F(z_{M,p}M,M)\geq F(2,2)\approx0.63$, yielding
\begin{align*}
 P(X\leq z_{M,p}M) = F(z_{M,p}M,M)= 1-p\geq 0.63\,,
\end{align*}
thus $0<p<0.37$.
\end{proof}

\begin{remark}
for fixed $0<p<0.37$, since
\begin{align*}
P(X\leq \overline{z}_{M,p}M)\geq P(X\leq z_{M,p}M) =1-p\,,
\end{align*}
for $X = \frac{NM\triangle_M^2}{\sigma^2}\sim\chi^2(M)$, the estimate
\begin{align}\label{important}
\triangle_M^2\leq\frac{\overline{z}_{M,p}\sigma^2}{N}
\end{align}
is satisfied with probability of at least $1-p$, i.e.,
\begin{align*}
P(\triangle_M^2\leq\frac{\overline{z}_{M,p}\sigma^2}{N})\geq 1-p\,.
\end{align*}
\end{remark}

One important thing is how to choose the regularization parameter $\alpha$ in the functional $\Phi$ so that the minimizer ca be one possible solution of the numerical differentiation problem. Our consideration is taking $\alpha = \frac{\bar{c}\sigma^2}{N}$ with a constant $\bar{c}$. This is motivated by the results in previous work in \cite{cjw07}. On one hand, variance describes the fluctuation level of the random variable, so choosing $\alpha$ be the same order with $\sigma^2$ is an intuitive consideration. On the other hand, based on the central limit theorem, taking the sample mean on a certain interval as the observation value, the corresponding error variance will be reduced according to the speed of $1/N$, thus $\alpha$ should also reduced.

Based on the above discussions, we will establish the convergence results. The important Sobolev inequality is necessarily be given before the convergence theorem.
\begin{lemma}\label{sobolev}
Let $-\infty\leq a<b\leq\infty,1\leq p<\infty$ and $0<\varepsilon_0<\infty$, $f$ is a function that has $m-$th order derivative in $(0,1)$. There exists a constant $K$ which depends on $\varepsilon_0$, $p$, and such that for every $\varepsilon$, $0<\varepsilon\leq\varepsilon_0$, $0\leq j<m$, we have
\begin{align*}
\int_0^1 |f^{(j)}|^p dt\leq K\varepsilon\int_0^1|f^{(m)}|^pdt + K\varepsilon^{\frac{-j}{m-j}}\int_0^1|f|^pdt\,.
\end{align*}
\end{lemma}

\begin{theorem}\label{main}
Suppose $y=y(x)\in H^k(0,1)$ is a function defined on $[0,1]$, given the noisy observation samples $\widetilde{y}_j= y(\hat x_j)+\eta_{j}$ with $1\leq j\leq L$ of the function $y(x)$ satisfying
 \begin{align}
\eta_i \sim N(0,\sigma^2)\,.
\end{align}
Dividing the samples into $M$ groups as in \eqref{G} and define the new grid $\triangle$ as in \eqref{newgrid}. Let $\widetilde{Y}_i$ with $1\leq i\leq M$ and statistics $\triangle_M^2$ are defined in \eqref{functional} and \eqref{triangle} respectively.  For fixed $0<p<0.37$, let $\bar{z}_{M,p}$ be the unique solution for the function $x\exp(1-x) = p^{2/M}$ on interval $(1,+\infty)$ defined in Lemma \ref{chi}, then $\triangle_M^2$ satisfies
\begin{align*}
\triangle_M^2
\leq\frac{\overline{z}_{M,p}\sigma^2}{N}
\end{align*}
with the probability of at least $1-p$. Let $e = f_*-y$, when choosing the regularization parameter $\alpha = \frac{\bar{c}\sigma^2}{N}$, the $L^2$ norm of the $e$ and $e^{(k)}$ satisfies the following estimates
\begin{align}
&\|e^{(k)}\|\leq \sqrt{\frac{2\overline{z}_{M,p}}{\bar c}+\frac{2Q^2h_M^2}{\bar{c}N\sigma^2}} + 2\|y^{(k)}\|\,,\label{estimate11}\\
&\|e\|\leq h\|e'\|+\sqrt{\frac{8\overline{z}_{M,p}\sigma^2+2\bar{c}\sigma^2\|y^{(k)}\|^2}{N}+\frac{8Q^2h_M^2}{N^2}}\label{estimate12}\,,
\end{align}
with the probability of at least $1-p$. Therefore, for fixed $0\leq j\leq k-1$, assume $\sigma^2>h_M^2/N$, pthe $L^2$ norm of $e^{(j)}$ satisfying the estimate
\begin{align}\label{final}
\|e^{(j)}\|_2\leq C_1 h_M^{k-j}+C_2\left(\frac{\sigma^2}{N}\right)^{(k-j)/2k}\,
\end{align}
with the probability of at least $1-p$. The constants $C_1$ and $C_2$ are two constants independent of $h_M$, $\sigma$ and $N$.
\end{theorem}
\begin{proof}
The \eqref{estimate11} and \eqref{estimate12} is directly obtained from \eqref{estimate1} and \eqref{estimate2} by taking the estimate $\triangle_M^2$ and $\alpha$. We then prove the \eqref{final} for the case $j=1$, i.e.,
\begin{align*}
\|e'\|_2\leq C_1h_M^{k-1}+C_2\left(\frac{\sigma^2}{N}\right)^{\frac{k-1}{2k}}\,,
\end{align*}
Taking $j=1$, $m=k$, $p=2$, $\varepsilon_0=1$ and $f=e$ in Lemma \ref{sobolev}, we can obtain
\begin{align*}
\|e'\|^2\leq K\varepsilon\|e^{(k)}\|^2+K\varepsilon^{-\frac{1}{k-1}}\|e\|^2\,.
\end{align*}
Without losing the generality, we assume that $\|e\|^{2(k-1)/k}\leq\varepsilon_0=1$. Taking $\varepsilon = \|e\|^{2(k-1)/k}$, it follows that
\begin{align*}
\|e'\|^2\leq K(\|e^{(k)}\|^2+1)\|e\|^{2(k-1)/k}\,.
\end{align*}
Referring to the estimates for both $\|e\|$ and $\|e^{(k)}\|$ in \eqref{estimate11} and \eqref{estimate12}, we have
\begin{align*}
\|e'\|^2\leq K\left(h\|e'\|+\sqrt{\frac{8\overline{z}_{M,p}\sigma^2+2\bar c\sigma^2\|y^{(k)}\|^2}{N}+\frac{8Q^2h_M^2}{N^2}}\right)^{2(k-1)/k}
\left(1+\frac{2\overline{z}_{M,p}}{\bar c}+\frac{2Q^2h_M^2}{\bar cN\sigma^2}+4\|y^{(k)}\|^2\right)\,.
\end{align*}
It is equivalently that
\begin{align*}
\|e'\|^k&\leq K^{k/2}\left(h_M\|e'\|+\sqrt{\frac{8\overline{z}_{M,p}\sigma^2+2\bar c\sigma^2\|y^{(k)}\|^2}{N}+\frac{8Q^2h_M^2}{N^2}}\right)^{k-1}
\left(1+\frac{2\overline{z}_{M,p}}{\bar c}+\frac{2Q^2h_M^2}{\bar cN\sigma^2}+4\|y^{(k)}\|^2\right)^{k/2}\\
&\leq K^*\left(1+\frac{2\overline{z}_{M,p}}{\bar c}+\frac{2Q^2h_M^2}{\bar cN\sigma^2}+4\|y^{(k)}\|^2\right)^{k/2}\\
&\qquad\cdot\left(\|e'\|^{k-1}h_M^{k-1}+\left(\frac{\left(8\overline{z}_{M,p}+2\bar c\|y^{(k)}\|^2\right)\sigma^2}{N}+\frac{8Q^2h_M^2}{N^2}\right)^{(k-1)/2}\right)\,,
\end{align*}
in which the constant $K^* := 2^{k-1}K^{k/2}$.

Next we will show that \eqref{final} can be obtained from the above estimate. We consider two cases.

{\bf Case1}: Assume $\|e'\|\leq K^*\left(1+\frac{2\overline{z}_{M,p}}{\bar c}+\frac{2Q^2h_M^2}{\bar cN\sigma^2}+4\|y^{(k)}\|^2\right)^{k/2}h_M^{k-1}$, since $\overline{z}_{M,p}$ can be bounded, yielding
\begin{align*}
\|e'\|&\leq K^*\left(K_1+K_2\frac{h_M^2}{N\sigma^2}\right)^{k/2}h_M^{k-1}\leq C_1h_M^{k-1}\,,
\end{align*}
in which the constants $K_1: = 1+\frac{2\overline{z}_{M,p}}{\bar c}+4\|y^{(k)}\|^2$, $K_2 = 2Q^2/\bar c$ and $C_1 = K^*(K_1+K_2)^{k/2}$. The last inequality is based on the assumption $\sigma^2>h_M^2/N$.

{\bf Case2}: Assume $\|e'\|> K^*\left(1+\frac{2\overline{z}_{M,p}}{\bar c}+\frac{2Q^2h_M^2}{\bar cpN\sigma^2}+4\|y^{(k)}\|^2\right)^{k/2}h_M^{k-1}$. Then we can take
\begin{align*}
r = \|e'\|-K^*\left(1+\frac{2\overline{z}_{M,p}}{\bar c}+\frac{2Q^2h_M^2}{\bar cN\sigma^2}+4\|y^{(k)}\|^2\right)^{k/2}h_M^{k-1}>0\,,
\end{align*}
then,
\begin{align*}
r^k&\leq \|e'\|^{k-1}r\\
&\leq \|e'\|^k-K^*\left(1+\frac{2\overline{z}_{M,p}}{\bar c}+\frac{2Q^2h_M^2}{\bar cN\sigma^2}+4\|y^{(k)}\|^2\right)^{k/2}h_M^{k-1}\|e'\|^{k-1}\\
&\leq K^*\left(1+\frac{2\overline{z}_{M,p}}{\bar c}+\frac{2Q^2h_M^2}{\bar cN\sigma^2}+4\|y^{(k)}\|^2\right)^{k/2}\left(\frac{\left(8\overline{z}_{M,p}+2\|y^{(k)}\|^2\right)\sigma^2}{N}+\frac{8Q^2h_M^2}{N^2}\right)^{(k-1)/2}\\
&\leq C_1
\left(K_3\frac{\sigma^2}{N}+K_4\frac{h_M^2}{N^2}\right)^{(k-1)/2}\\
&\leq C_2^k\frac{\sigma^{k-1}}{N^{(k-1)/2}}\,,
\end{align*}
where the constants $K_3 = 8\overline{z}_{M,p}+2\|y^{(k)}\|^2$, $K_4 = 8Q^2$ and $C_2^k = C_1\max(K_3,K_4)^{(k-1)/2}$.
Consequently,
\begin{align*}
r \leq C_2\left(\frac{\sigma^2}{N}\right)^{(k-1)/2k}\,.
\end{align*}
The inequality \eqref{final} for $j=1$ is proved. By a similar method, we can prove for any $0\leq j<k$, it holds that
\begin{align*}
\|e^{(j)}\|_2\leq C_1 h_M^{k-j}+C_2\left(\frac{\sigma^2}{N}\right)^{(k-j)/2k}\,.
\end{align*}
We will not give the detailed proof here.

\end{proof}

\begin{remark}\label{remark}
For particular $k=2$, the estimate \eqref{final} becomes
\begin{align*}
\|e'\|_2\leq C_1 h_M+C_2\left(\frac{\sigma^2}{N}\right)^{1/4}\,.
\end{align*}
When $\sigma$ is fixed, pay attention that $h_M = N/J$, this motivate us that $N = \mathcal{O}(J^{4/5})$ and $M = \mathcal{O}(J^{1/5})$ is the optimal choice.
\end{remark}
\section{Algorithm}
\setcounter{equation}{0}
We consider the simple case $k=2$. The other cases can be treated in a similar way. Since $f_*$ is a piece wise polynomial of degree four, we assume that
\begin{align}\label{fstar}
f_*(x)  = a_i+b_i(x-x_i)+c_i(x-x_i)^2+d_i(x-x_i)^3+e_i(x-x_i)^4,\quad\textrm{for}\ x\in[x_i,x_{i+1})\,,
\end{align}
where there are $5M$ constants $a_i$, $b_i$, $c_i$, $d_i$, $e_i$ for $i=0,\ldots,M-1$.

From our reconstruction, we have
\begin{align*}
\frac{c_{i-1}-2c_i+c_{i+1}}{3h} = 3(e_{i-1}+e_i)h_M,\quad\textrm{for}\ i=1,\ldots,M-1\,.
\end{align*}
Since $f''(0)=f''(1)=0$ yields $c_0 = c_M=0$, denote $c = (c_1,\ldots,c_{M-1})^T$ and $e = (e_0,e_1,\ldots,e_{M-1})^T$, we have
\begin{align}\label{c}
c = 6h^2(A_1)^{-1}B_1e : = Ce\,,
\end{align}
with
\begin{align}\label{c1}
A_1 = \left(
\begin{matrix}
-2 &1 &0 &\cdots\\
1 &-2 &1 &\cdots\\
\vdots &\vdots &\vdots &\vdots\\
\cdots &1 &-2 &1\\
\cdots &\cdots &1 &-2\\
\end{matrix}\right)_{(M-1)\times(M-1)} B_1 = \left(
\begin{matrix}
1 &1 &0 &\cdots\\
0 &1 &1 &\cdots\\
\vdots &\vdots &\vdots &\vdots\\
\cdots &\cdots &1 &1\\
\end{matrix}\right)_{(M-1)\times M}\,.
\end{align}

Then, since
\begin{align*}
d_{i-1}  = \frac{1}{3h_M}\left(c_i-c_{i-1}-6e_{i-1}h_M^2\right)\,,
\end{align*}
denote $d = (d_0,d_1\ldots,d_{M-1})^T$, it follows that
\begin{align}\label{d}
d = \frac{1}{3h}Tc-2he: = De\,,\textrm{where}\quad T = \left(
\begin{matrix}
1 &0 &0 &\cdots\\
-1 &1 &0 &\cdots\\
\vdots &\vdots &\vdots &\vdots\\
\cdots &\cdots &0 &-1\\
\end{matrix}\right)_{M\times(M-1)}
\end{align}

Next, since
\begin{align*}
a_{i-1}-2a_i+a_{i+1} = (c_{i-1}+c_i)h_M^2+(2d_{i-1}+d_i)h_M^3+(3e_{i-1}+e_i)h_M^4,\quad\textrm{for}\ i=1,\ldots,M-1\,,
\end{align*}
this combine with $a_0 = f(0)$ and $a_M = f(1)$, denote $a = (a_1,\ldots,a_{M-1})^T$, it follows that
\begin{align}\label{a}
A_1 a = h_M^2Pc+h_M^3Qd+h_M^3Re-\frac{1}{h_M}v_1: = Ae-\bar{v}_1\,,
\end{align}
where
\begin{align}\label{d1}
&P = \left(
\begin{matrix}
1 &0 &0 &\cdots\\
1 &1 &0 &\cdots\\
\vdots &\vdots &\vdots &\vdots\\
\cdots &\cdots &1 &1\\
\end{matrix}\right)_{(M-1)\times(M-1)}
Q = \left(
\begin{matrix}
2 &1 &0 &\cdots\\
0 &2 &1 &\cdots\\
\vdots &\vdots &\vdots &\vdots\\
\cdots &\cdots &2 &1\\
\end{matrix}\right)_{(M-1)\times M} \nonumber\\
&R = \left(
\begin{matrix}
3 &1 &0 &\cdots\\
0 &3 &1 &\cdots\\
\vdots &\vdots &\vdots &\vdots\\
\cdots &\cdots &3 &1\\
\end{matrix}\right)_{(M-1)\times M}
v_1 = \left(\begin{matrix}
f(0)\\
0\\
\vdots\\
0\\
f(1)\\
\end{matrix}\right)_{(M-1)\times 1}
\end{align}

In addition, since
\begin{align*}
b_{i-1} = \frac{1}{h_M}(a_i-a_{i-1}) - c_{i-1}h_M-d_{i-1}h_M^2-e_{i-1}h_M^3\,,\quad\textrm{for}\ i=1\ldots M-1
\end{align*}
Therefore,
\begin{align}\label{b}
b = \frac{1}{h_M}Ta-h_M\left(\begin{matrix}
0\\
c\\
\end{matrix}\right)-h_M^2d-h_M^3e-\frac{1}{h_M}v_2: = Be-\bar{v}_2\,,
\end{align}
where
\begin{align}\label{b1}
R = \left(
\begin{matrix}
1 &0 &0 &\cdots\\
-1 &1 &0 &\cdots\\
\vdots &\vdots &\vdots &\vdots\\
\cdots &\cdots &0 &-1\\
\end{matrix}\right)_{M\times (M-1)} v_2 = \left(\begin{matrix}
f(0)\\
0\\
\vdots\\
0\\
-f(1)\\
\end{matrix}\right)_{M\times 1}
\end{align}

Finally, let $e=(e_0,e_1,\ldots,e_{M-1})^T$, we have
\begin{align}\label{e}
24\alpha e = \widetilde{Y}-\left(\begin{matrix}
f(0)\\
a\\
\end{matrix}\right)-\frac{h_M}{2}b-\frac{h_M^2}{3}\left(\begin{matrix}
0\\
c\\
\end{matrix}\right)-\frac{h_M^3}{4}d -\frac{h_M^4}{5}e\,.
\end{align}
This equivalent to
\begin{align}\label{e1}
\left(\left(24\alpha+\frac{h_M^4}{5}\right)I + \frac{h_M}{2}B+\frac{h_M^3}{4}D\right)e+\left(\begin{matrix}
0\\
\left(A+\frac{h_M^2}{3}C\right)e\\
\end{matrix}\right) = \widetilde{Y}+\left(\begin{matrix}
-f(0)\\
\bar{v}_1\\
\end{matrix}\right)+\frac{h_M}{2}\bar{v}_2\,.
\end{align}
By solving \eqref{e1} and using the matrix relations, we can get the coefficient vectors.
\section{Numerical example}
\setcounter{equation}{0}
In this section, some numerical examples are provided to illustrate computational performance of the method. The regularization parameter is selected by {\it a posteriori} choice L-curve criterion. We divide $[0,1]$ into $L = 1000$ equal subintervals. For a given function $y = y(x)= x^3+2x^2-0.5x+1$, we add random noise at each points with normal distribution $N(0,\sigma^2)$ with $\sigma^2=0.2$ and generate corresponding noisy observations, see the Figure \ref{figure1}~(left). We divided the points into $M=10$ groups, so there are $N=100$ points in every group. We generate the $\{\widetilde{Y}_i\}_{i=1}^M$ be the sample mean of these $N$ observations, see Figure \ref{figure1}~(right).
\begin{figure}[htp]
  \begin{center}
  \includegraphics[width=6cm,height=6.5cm]{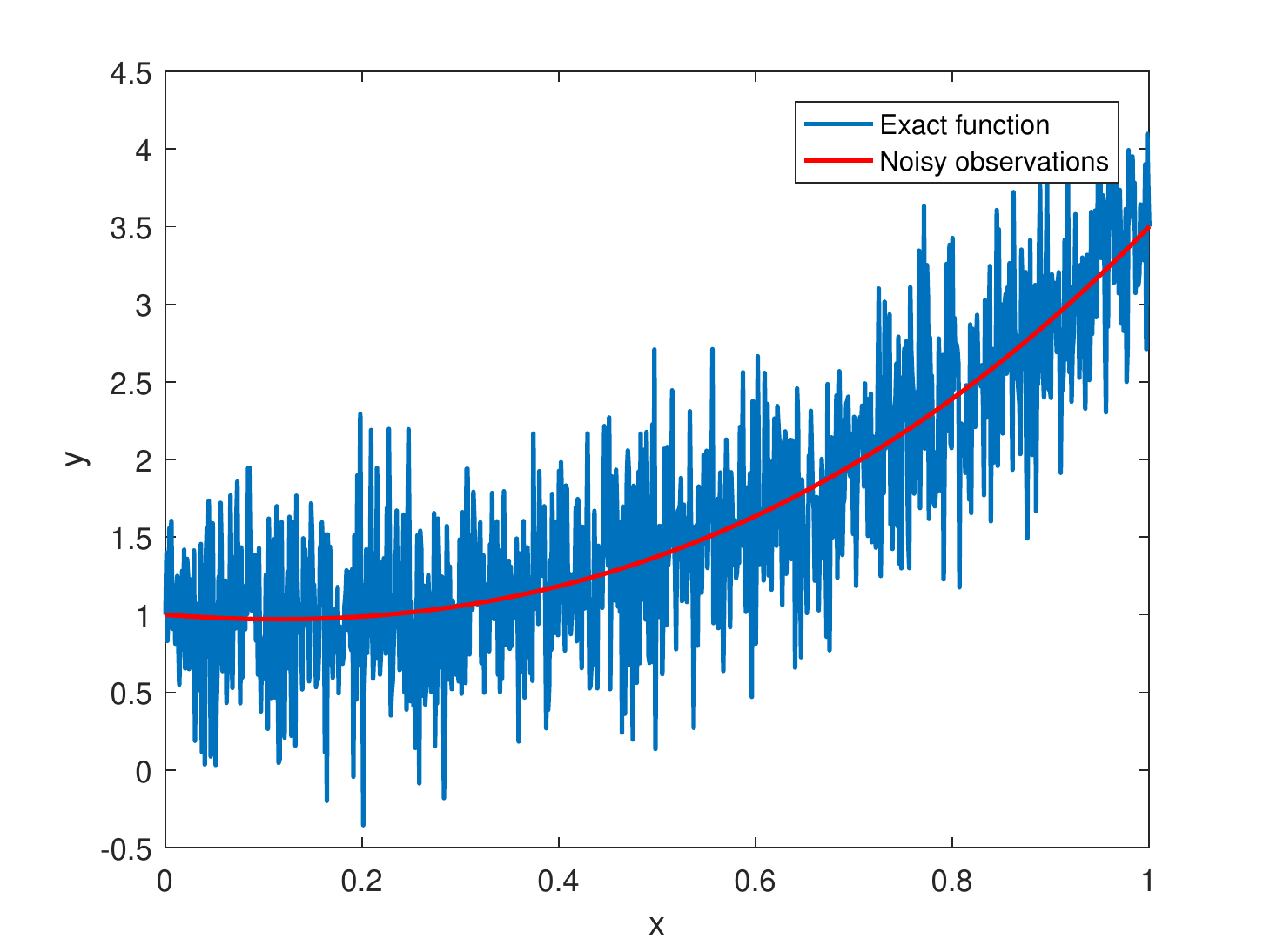}
  \includegraphics[width=6cm,height=6.5cm]{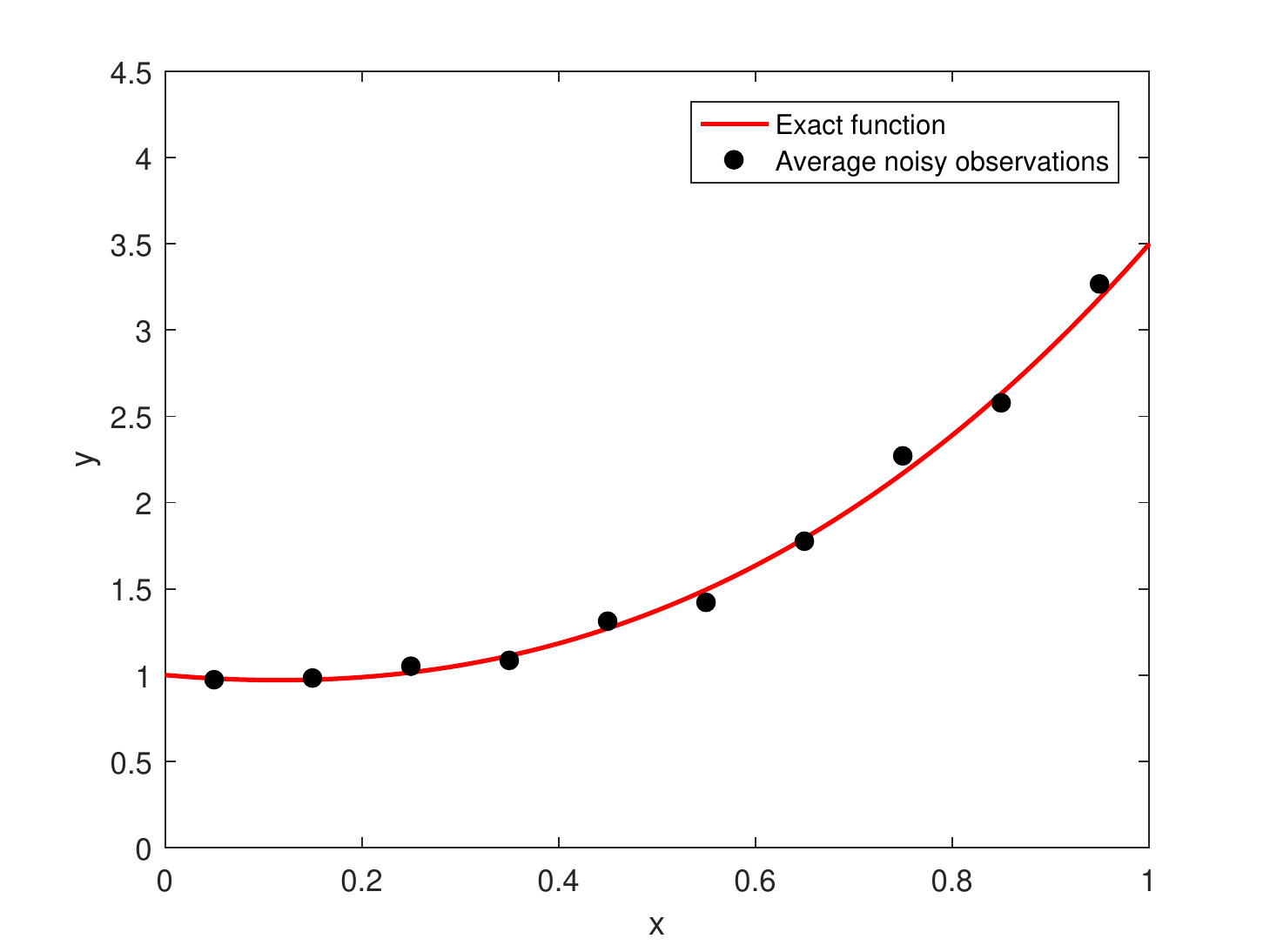}
  \end{center}
  \caption{The noisy observations $\sigma^2=0.2$ (left), the average noisy observations with $M=10$ (right). \label{figure1}}
\end{figure}

In our numerical example, we fix $k=2$ and do the algorithm in section 4, the regularization parameter $\alpha$ was suggested be chosen as $\alpha = \frac{\bar{c}\sigma^2}{N}$, now we illustrate the performance of the heuristic parameter L-curve strategy for determining the constant $\bar\alpha$. We plot the curve $(\log(\|(f_*)'\|_{L^2(0,1)}^2),\log(\textrm{residual}))$ with $$\textrm{residual}:=\sqrt{\frac{1}{M}\sum_{i=1}^M\left(\widetilde{Y}_i-M_i(f)\right)^2}\,.$$
The curve indeed looks like the letter ``L'', and we can get the corresponding constant $\bar c=0.0239$.
\begin{figure}[htp]
  \begin{center}
  \includegraphics[scale=0.7]{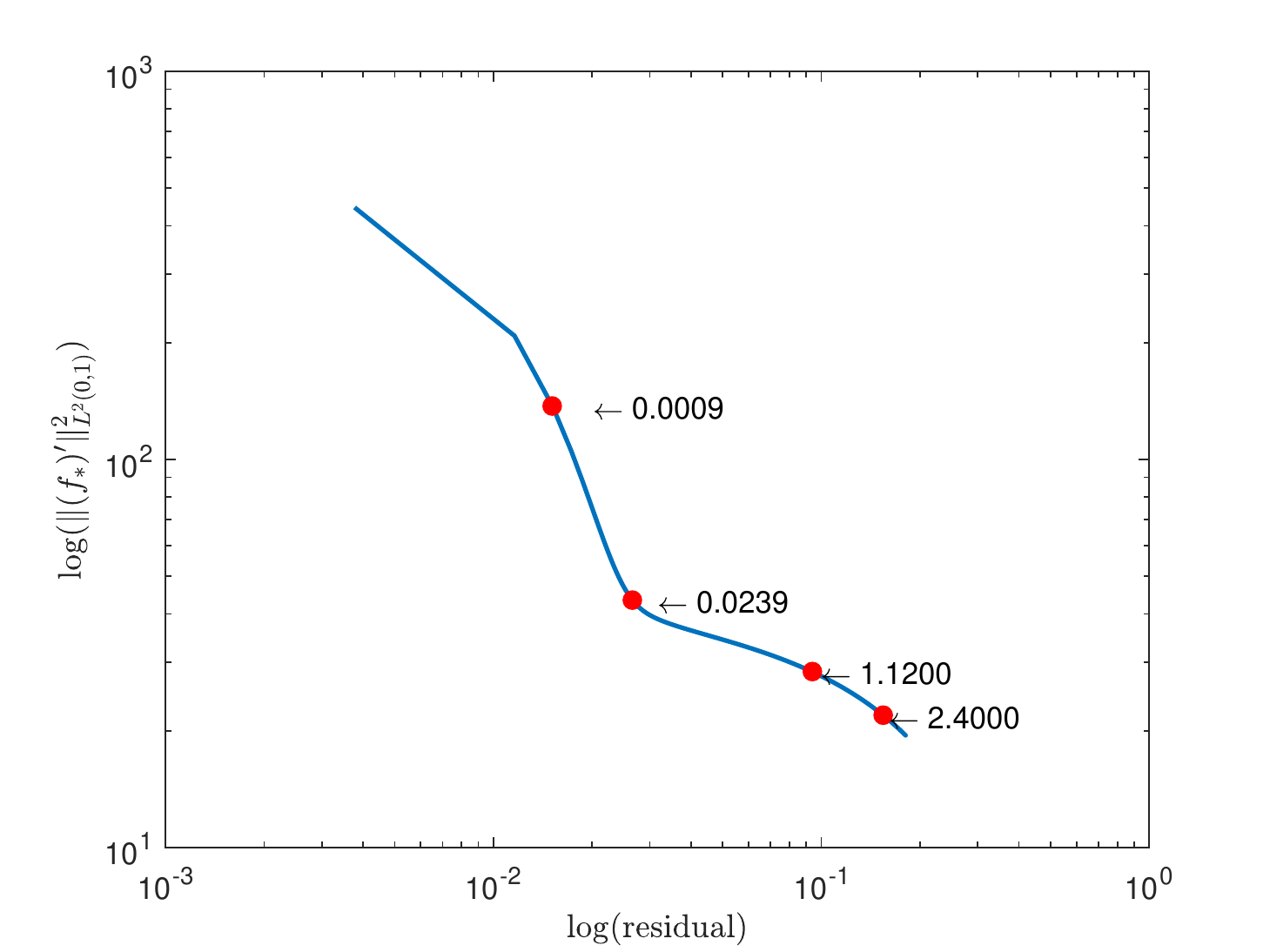}
  \end{center}
  \caption{The loglog figure of the Lcurve. \label{figure2}}
\end{figure}
Then, we present the computational performance of the method and the regularization parameter is chosen in terms of L-curve method, the reconstruction for $y(x)$ and $y'(x)$ are showed in Figure \ref{figure3} respectively. Here and in what follows, the blue curve means the exact function or the exact first derivative, while the red curve means the reconstructions.
\begin{figure}[htp]
  \begin{center}
  \includegraphics[width=6cm,height=6.5cm]{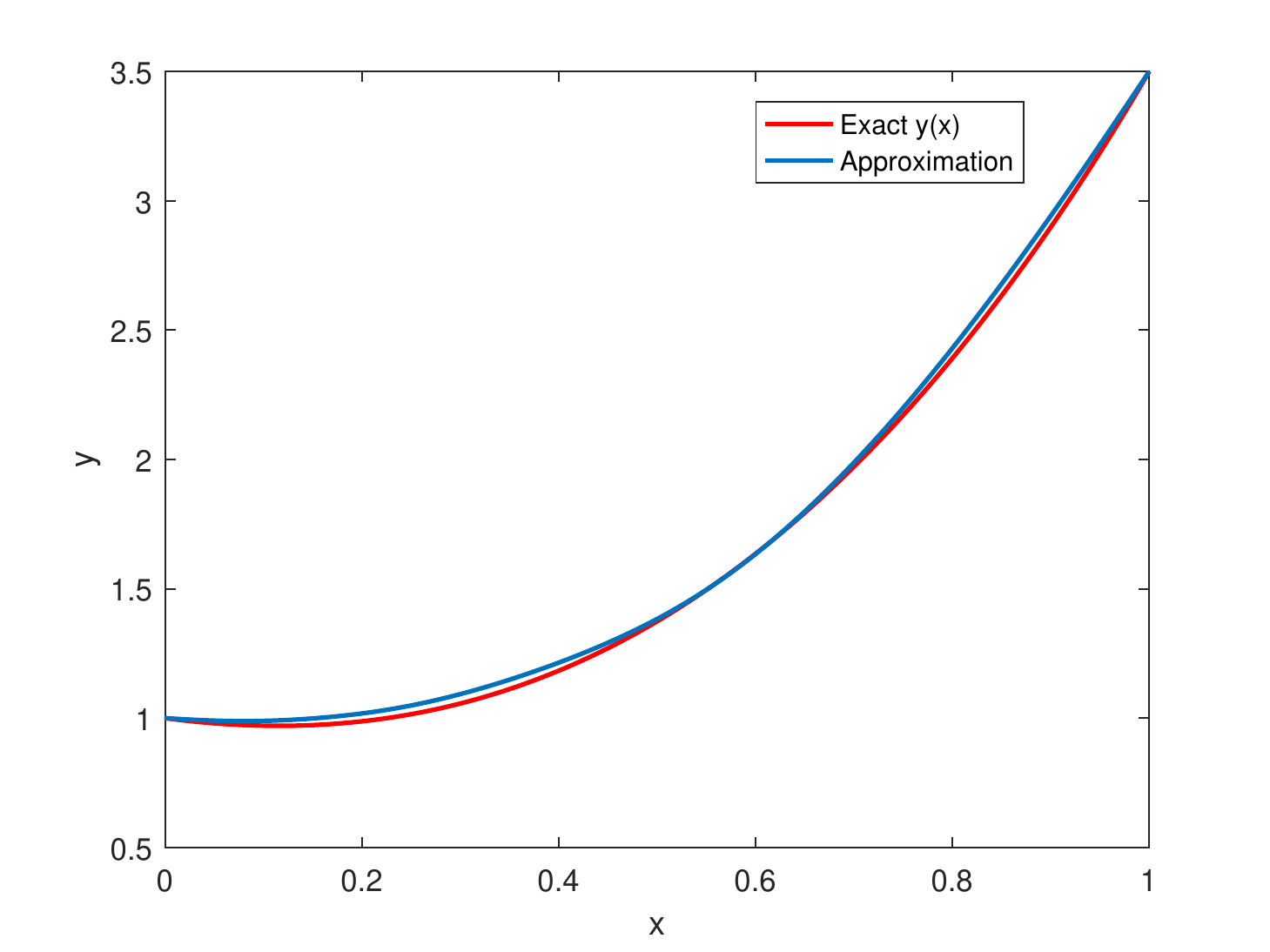}
   \includegraphics[width=6cm,height=6.5cm]{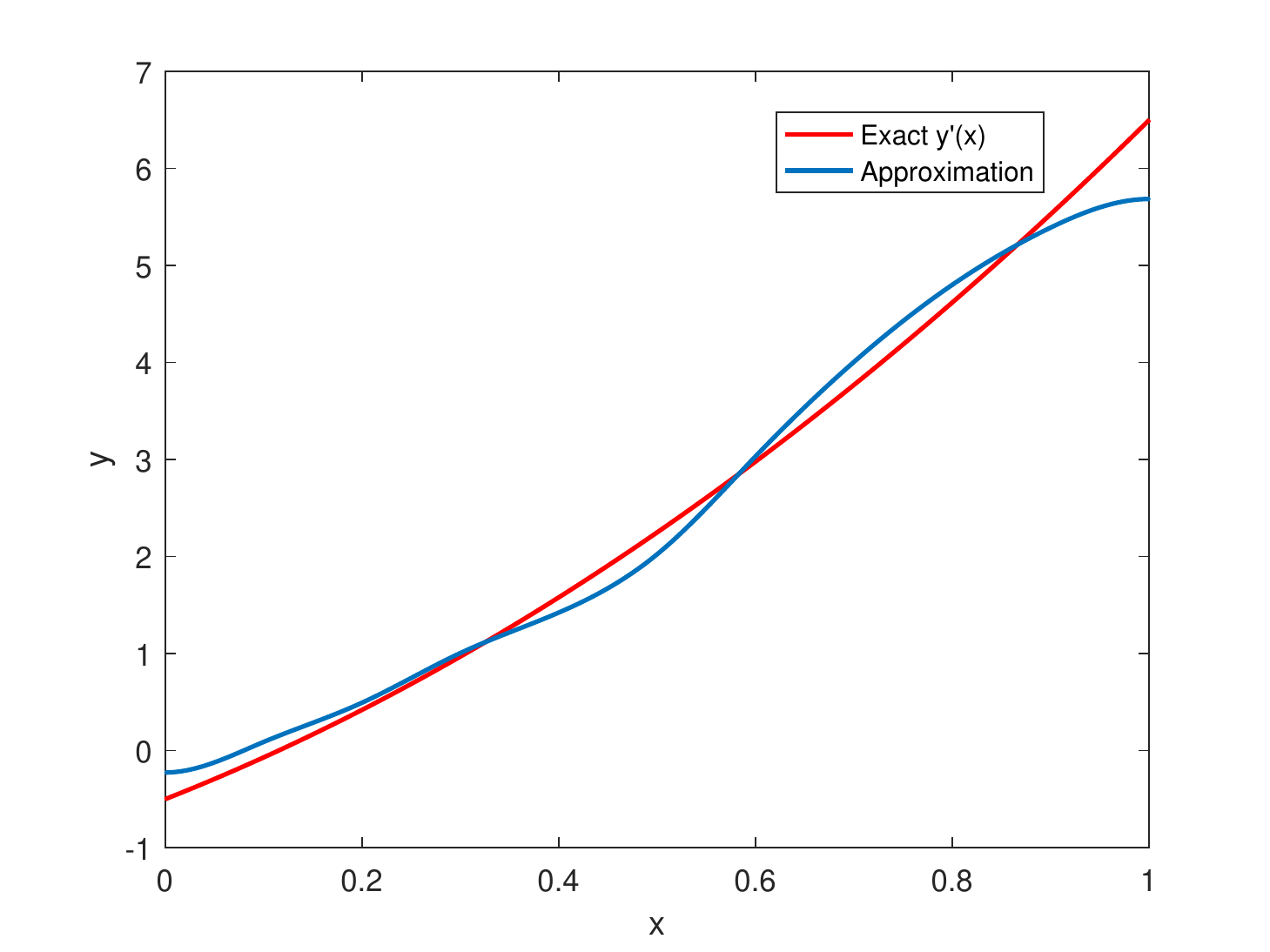}
  \end{center}
  \caption{The reconstruction of $y(x)$ (left) and the reconstruction of $y'(x)$ (right) with $M=10$. \label{figure3}}
\end{figure}

Next, we compare the reconstruction results under different choices for $M$, recalling the Remark \ref{remark}, the optimal choice is $M = \mathcal{O}(J^{1/5})$, which means $M$ is approximately $5$. we choose $M=5,10,50,100,200$ for comparison, in all situations, we fix the same constant $\bar c = 0.0239$ and thus, due to different $N$, the $\alpha = \bar c\sigma^2/N$ is also different. The detailed information was shown in Table \ref{table1}. It is also necessary to notice that, the computational complexity was based on the value of $M$, smaller $M$ represents smaller matrices sizes and thus cheaper computational costs.
\begin{table}[h]
  \caption{The reconstruction comparison for different choices for $M$}\label{table1}
    \begin {center}
\begin{tabular}{llcccc}
      \hline
$M$  & $N$     &  $\|y-f_*\|_{L^2(0,1)}$  & $ \|y-f_*\|_{l^\infty}$ &  $\|y'-f_*'\|_{L^2(0,1)}$  & $\|y'-f_*'\|_{l^\infty}$          \\
\hline
5    & 200    &  0.020805   &  0.036963  & 0.166882 & 0.745254 \\
10   & 100    &  0.027061   &  0.045420  & 0.211428 & 0.815453 \\
50   & 20     &  0.040842   &  0.067272  & 0.249623 & 1.023243 \\
100  & 10     &  0.054420   &  0.085524  & 0.287166 & 1.150439 \\
200  & 5     &  0.079110   &  0.116156  & 0.353859 & 1.333828\\
   \hline
    \end{tabular}\\[5mm]
    \end{center}
\end{table}

Finally, we compare our method with the previous algorithm in \cite{cjw07}, in which they discussed the approximation provided that the noisy observations satisfying $\|y(x_j)-y^\delta_j\|\leq\delta$ for $1\leq j\leq L$, and the noiselevel $\delta$ should be known {\it a priorily} and the regularization parameter was suggested be chosen as $\alpha = \delta^2$. In this example, since the noise was given as a normal distribution with $\sigma^2$ be the variance, it is appropriate to choose $\alpha = \sigma^2$.
\begin{figure}[htp]
  \begin{center}
  \includegraphics[width=6cm,height=6.5cm]{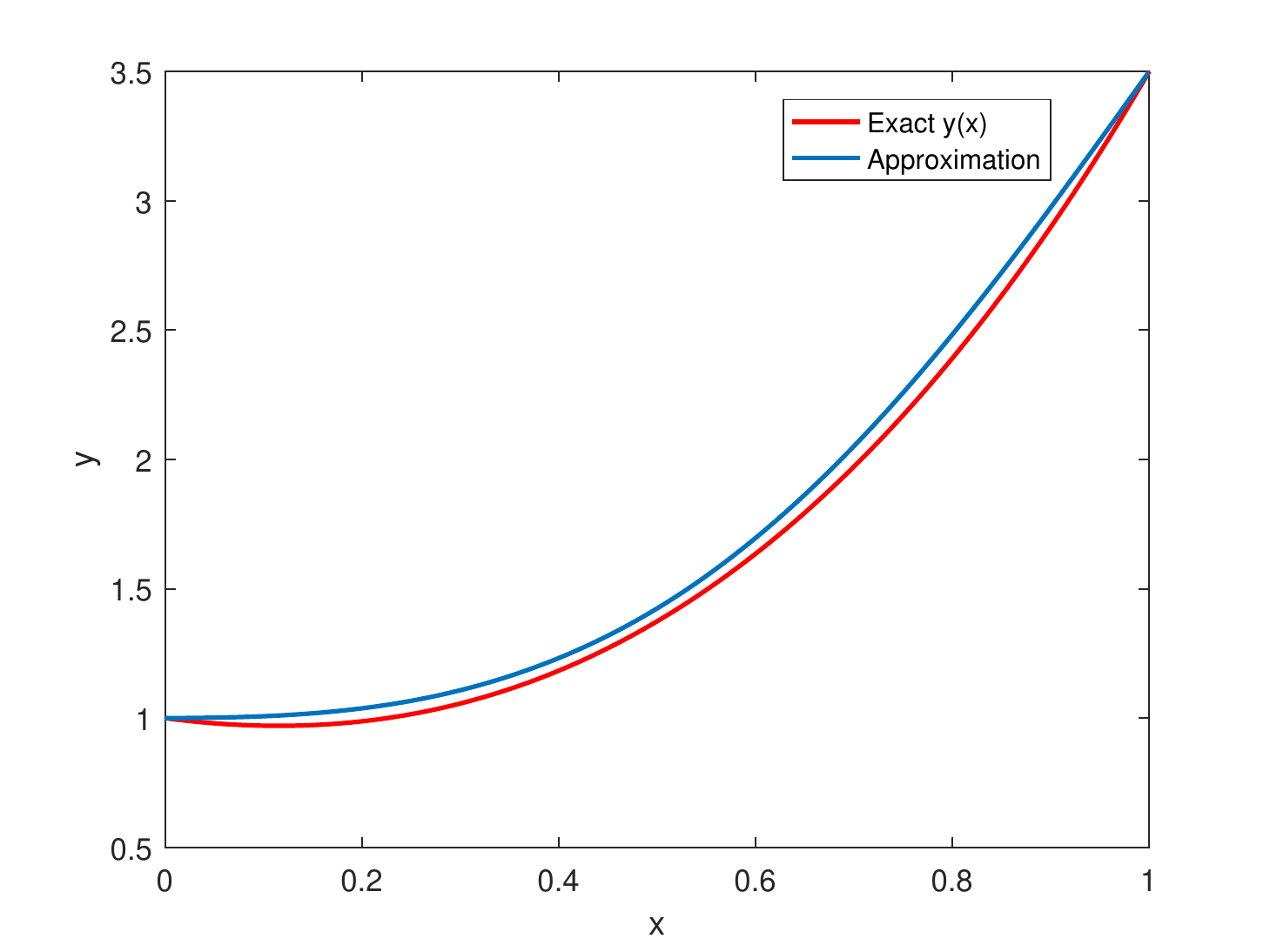}
   \includegraphics[width=6cm,height=6.5cm]{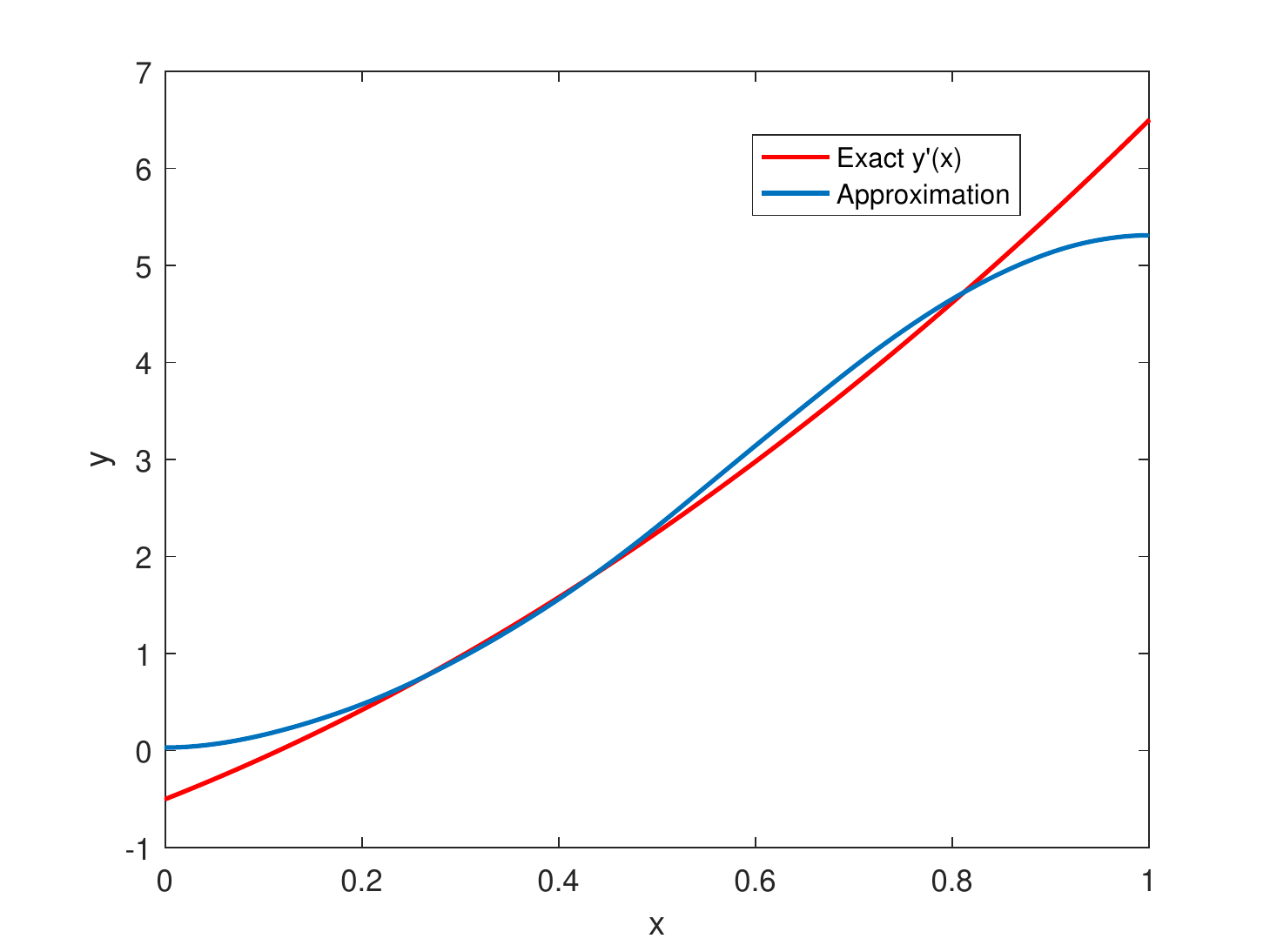}
  \end{center}
  \caption{The reconstruction of $y(x)$ (left) and the reconstruction of $y'(x)$ (right) using the method in \cite{cjw07}. \label{figure5}}
\end{figure}
The error of $\|y-f_*\|_{L^2(0,1)}$, $ \|y-f_*\|_{l^\infty}$, $\|y'-f_*'\|_{L^2(0,1)}$ and $\|y'-f_*'\|_{l^\infty}$ are $0.059733$, $0.092025$, $0.299401$ and $1.179538$ respectively. However, the computational complexity is much higher, since the matrices involved in calculation is $1000\times1000$.

The advantages of our method will be reflected when the observation data volume is very large. We provide the second example when the given function is $y(x) = 1+10x^2(1-x)^2$. Assume $L=10^6$, we add random noise at each points with normal distribution $N(0,\sigma^2)$ with $\sigma^2=0.25$ and generate corresponding noisy observations, see the Figure \ref{figurelarge1}. This time, we divided the points into $M=10$ groups, with $N=10^5$ points in every group, and the regularization parameter $\alpha$ was chosen be $\alpha = \sigma^2/N$ with constant $\bar\alpha=1$ for simplicity. The reconstruction results were plotted in Figure \ref{figurelarge1}, it is very satisfactory, and the computational complexity is very low. However, since the matrix size is too large to exceeds the capacity of the Matlab, the method by \cite{cjw07} was failed to be utilized.
\begin{figure}[htp]
  \begin{center}
  \includegraphics[width=6cm,height=6.5cm]{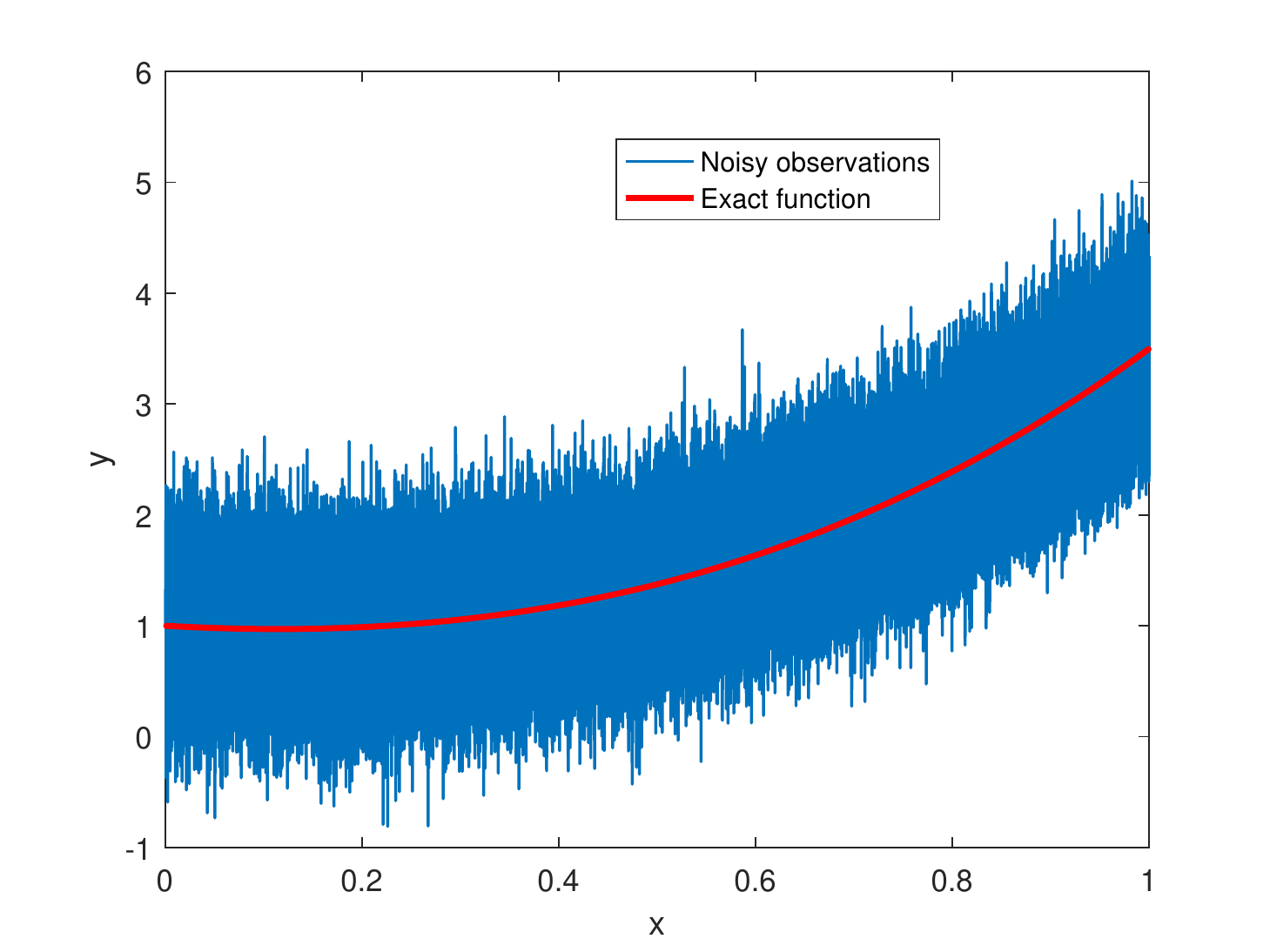}
  \includegraphics[width=6cm,height=6.5cm]{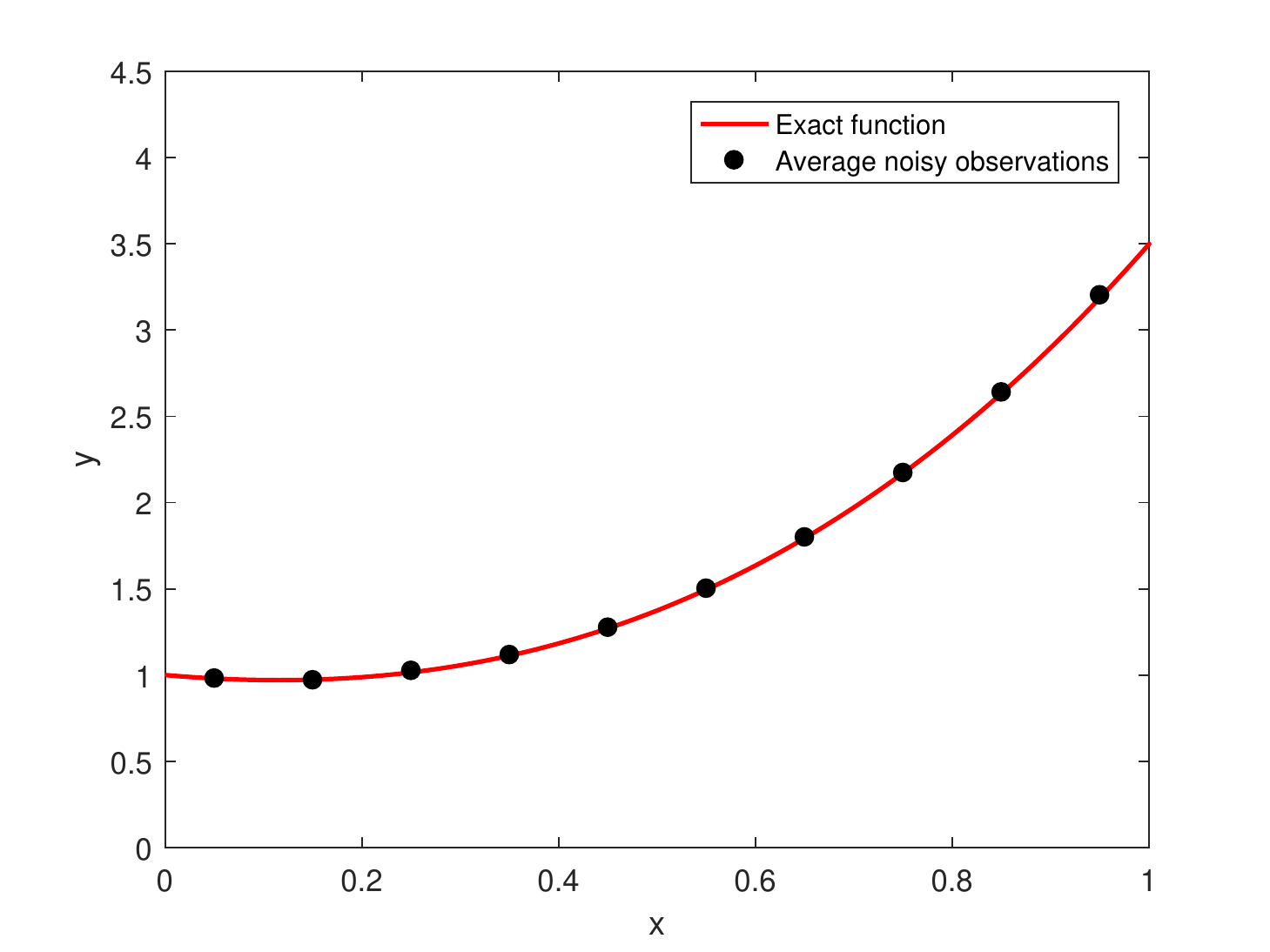}
  \includegraphics[width=6cm,height=6.5cm]{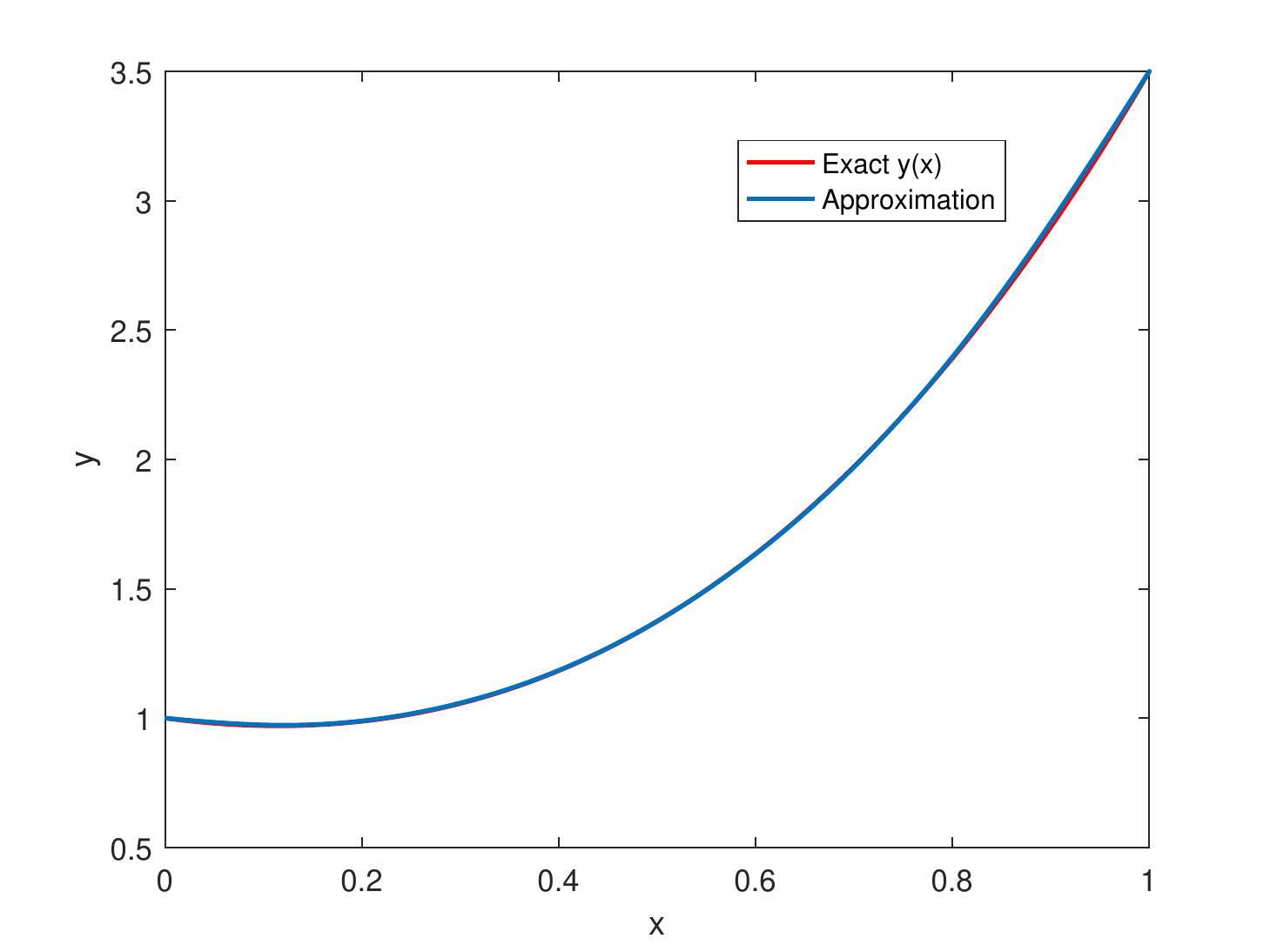}
  \includegraphics[width=6cm,height=6.5cm]{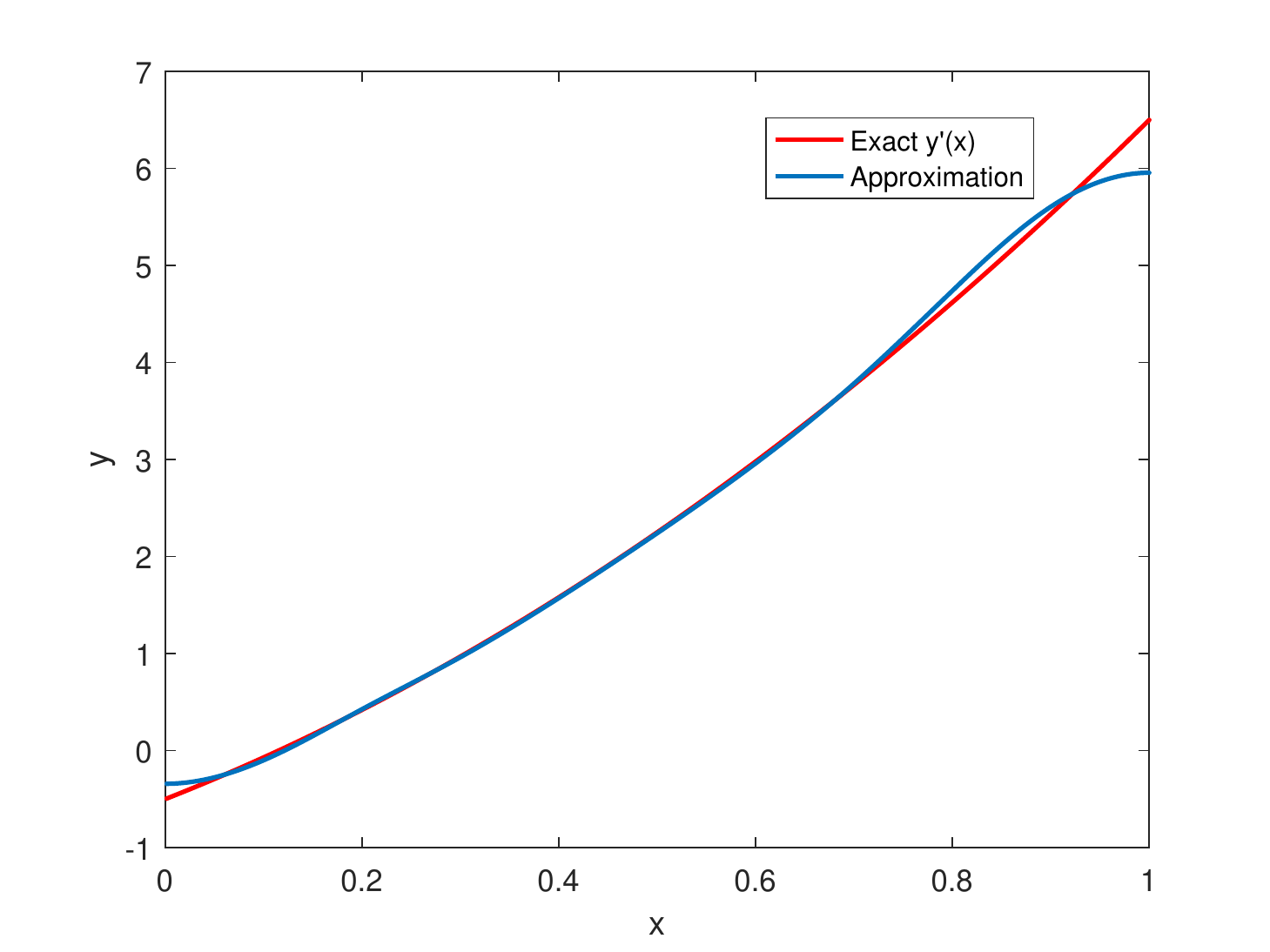}
  \end{center}
  \caption{The noisy observations, the average noisy observations with $M=10$, the approximation of $y(x)$ and the approximation of the $y'(x)$. \label{figurelarge1}}
\end{figure}
\section{Acknowledgement}
J. Cheng is supported by the NSFC (No.11971121), M. Zhong is supported by the NSFC (No. 11871149) and supported by Zhishan Youth Scholar Program of SEU.

  \end{document}